\documentclass[12pt,times]{amsart}

\usepackage{amsfonts, amsthm, amsmath, amssymb, times, amscd, enumerate, enumitem, verbatim, graphics, color, graphicx, cite, tikz, bm}
\usepackage{hyperref}
\hypersetup{colorlinks=false}

\usepackage[margin=1.5in]{geometry}

\numberwithin{equation}{section}

\newtheorem{theorem}{Theorem}[section]
\newtheorem{lemma}[theorem]{Lemma}
\newtheorem{proposition}[theorem]{Proposition}

\theoremstyle{definition}
\newtheorem*{ack}{Acknowledgements}
\newtheorem{remark}[theorem]{Remark}
\newtheorem{definition}[theorem]{Definition}

\renewcommand{\phi}{\varphi}
\renewcommand{\rho}{\varrho}

\renewcommand{\leq}{\leqslant}

\renewcommand{\geq}{\geqslant}

\newcommand{\bs}{\boldsymbol}

\renewcommand{\c}{\mathbf{c}}

\newcommand{\ve}{\varepsilon}

\DeclareMathOperator{\supp}{supp}

\DeclareMathOperator{\cond}{cond}

\newcommand{\sumstar}{\sideset{}{^*}\sum}
\newcommand{\jacobi}[2]{\left(\frac{#1}{#2}\right)}

\renewcommand{\mod}[1]{\hspace{-2.9mm}\pmod{#1}}

\begin{document}

\date{\today}

\title[Counting points on quadrics with arithmetic weights]{Counting integer points on quadrics with arithmetic weights}

\author{
 V. Vinay Kumaraswamy}

\address{School of Mathematics\\
University of Bristol\\ Bristol\\ BS8 1TW
\\ UK}

\email{vinay.visw@gmail.com}

\begin{abstract}
Let $F \in \mathbf{Z}[\bs{x}]$ be a diagonal, non-singular quadratic form in $4$ variables. Let $\lambda(n)$ be the normalised Fourier coefficients of a holomorphic Hecke form of full level. We give an upper bound for the problem of counting integer zeros of $F$ with $|\bs{x}| \leq X$, weighted by $\lambda(x_1)$. 
\end{abstract}

\maketitle

\thispagestyle{empty}

\setcounter{tocdepth}{1}
\tableofcontents
\section{Introduction}\label{}
The study of averages of arithmetic functions along thin sequences is a central topic in analytic number theory. For instance, the sum $\textstyle\sum_{n \leq X}a(p(n))$, where $p(n)=n^2+bn+c$ is an integer polynomial, and $a(n)$ are Fourier coefficients of automorphic forms, has been widely studied. For this sum, Hooley ~\cite{H} established an asymptotic formula with a power-saving error term when $a(n) = \tau(n)$, the divisor function, and $p(n)$ is irreducible. The case when $a(n)$ are Fourier coefficients of cusp forms was first settled by Blomer ~\cite{B}, and later refined by Templier and Tsimerman ~\cite{TT}. However, the analogous sum over the primes, i.e. the sum $\textstyle\sum_{n \leq X}\Lambda(p(n))$, where $\Lambda(n)$ is the von Mangoldt function, is much harder to estimate, and this is a long standing open problem. 

Mean values of arithmetic functions over polynomials of higher degree are poorly understood; obtaining an asymptotic formula for the sum $\textstyle\sum_{n \leq X} \tau(n^3+2)$ would represent a significant breakthrough in the subject. However, in the case of polynomials in more than variable, several results have been established. Among the most striking results in this regime are by Friedlander and Iwaniec ~\cite{FI}, on the existence of infinitely many primes of the form $x^2+y^4$, and by Heath-Brown ~\cite{HB01}, on primes of the form $x^3+2y^3$. 

Analogously, for the divisor function, sums of the form $\textstyle\sum_{m,n 
\leq X}\tau(|B(m,n)|)$, where $B(u,v)$ is an integral binary form of degree $3$ or $4$, have been investigated by several authors. For irreducible binary cubic forms, Greaves ~\cite{G2} gave an asymptotic formula for the aforementioned sum, and the sum over irreducible quartic forms was handled by Daniel ~\cite{D}. The case when $B(m,n)$ is not irreducible has also been considered; for example, such sums have been of much interest in problems relating to Manin's conjecture for del Pezzo surfaces. See ~\cite{Br}, where cubic forms are considered, and ~\cite{BB1}, ~\cite{BB2}, ~\cite{BT} and ~\cite{HB02} that treat the case of quartic forms.

Continuing in the same vein as the aforementioned results, the methods used in this paper can be used to prove a result of the following type,
\begin{equation}\label{eq:icanprovethis}
\sum_{m,n \leq X}r(Am^2+Bn^2)\lambda_f(m) \ll_{f,A,B} X^{2-\delta},
\end{equation}
for some $\delta > 0$, where $A$ and $B$ are non-zero integers, $\lambda_f(n)$ are normalised Fourier coefficients of a holomorphic Hecke cusp form $f$ of full level and $r(n)$ is the number of representations of an integer as a sum of two squares. In principle, this corresponds to the case where a cubic form $B(m,n)$ splits over $\mathbf{Q}$ as the product of a linear and a quadratic form\footnote{In this case, $B(m,n)=(Am^2+Bn^2)m$}.

Although we have stated~\eqref{eq:icanprovethis} with the $r$-function, our methods could potentially be adapted to deal with the divisor function. It is worth emphasising that existing results on divisor sums over binary cubic and quartic forms have largely relied on arguments involving the geometry of numbers, and one cannot expect to be able to establish~\eqref{eq:icanprovethis} by relying solely on such methods. Instead, we will draw from techniques in the theory of automorphic forms. Next, we move to our main theorem.

Let $F(\bs{x}) \in \mathbf{Z}[\bs{x}]$ be an integral quadratic form in four variables and let $w(\bs{x}) \in C_0^{\infty}(\mathbf{R}^4)$ be a smooth function with support in $[1/2,2]^4$. Let $$N_F(X) = \sum_{F(\bs{x})=0}w\left(\frac{\bs{x}}{X}\right),$$ count integral solutions to $F=0$ of height less than $X$. As $X \to \infty$, Heath-Brown ~\cite[Theorems 6,7]{HB} established an asymptotic formula for $N_F(X)$ with a power-saving error term (see also recent work of Getz ~\cite{G} where this is refined and a second order main term is given). Now, given an arithmetic function $a(n):\mathbf{N} \to \mathbf{C}$, it is natural to ask if we can count solutions to $F=0$ in which one of the variables is weighted by $a(n)$. More precisely, let 
\begin{equation*}\label{eq:nwfa} N(a;X) = N_F(a;X) = \sum_{F(\bs{x})=0}w\left(\frac{\bs{x}}{X}\right)a(x_1),\end{equation*} 
where $F$ and $w$ are as above. For instance, if $a = \Lambda$ then $N(a;X)$ counts weighted solutions to $F=0$ where one of the co-ordinates is prime. The inhomogeneous case (i.e. counting solutions to $F(\bs{x})=N$ for non-zero $N$), however, has been well-studied. Tsang and Zhao ~\cite{TZ} showed that every sufficiently large integer $N \equiv 4 \pmod{24}$ can be written in the form $p_1^2 + P_2^2 + P_3^2 + P_4^2$, where $p_1$ is a prime, and each $P_i$ has at most $5$ prime factors.  

In this work, we investigate the case where the $a(n)$ are Fourier coefficients of a holomorphic cusp form. Suppose that a holomorphic cusp form $f(z)$ has Fourier expansion $f(z) =  \sum_{n=1}^{\infty} \lambda(n)n^{\frac{k-1}{2}}e(nz),$ and then set $a(n) = \lambda(n)$. Our main result is 
\begin{theorem}\label{mainthm}
Let $F(\bs{x}) \in \mathbf{Z}[\bs{x}]$ be a non-singular diagonal quadratic form in $4$ variables, and let $w$ be a smooth function with compact support in $[1/2,2]^4$. Let $\lambda(n)$ be the normalised Fourier coefficients of a holomorphic Hecke cusp form $f$ of full level. Then for all $\ve > 0$ we have
\begin{equation*}
N(\lambda;X) \ll_{\ve,f,F,w} X^{2-\frac{1}{6}+\ve}.
\end{equation*}
\end{theorem}
To obtain~\eqref{eq:icanprovethis}, the above result is applied to the quadratic form $F(\bs{x}) = Ax_1^2+Bx_2^2-x_3^2-x_4^2$. 

From Heath-Brown's work on estimating $N_F(X)$ and Deligne's bound for $\lambda(n)$, we obtain the `trivial' bound $N(\lambda;X) \ll_{\ve} X^{2+\ve}$. Consequently, Theorem ~\ref{mainthm} detects cancellation owing to the sign changes of $\lambda(n)$. Moreover, although we shan't pursue it here, our method of proof allows us to handle slightly more general quadratic forms of the shape $A_1x_1^2+Q(x_2,x_3,x_4)$. It would also be very interesting to obtain an analogue of Theorem ~\ref{mainthm} with two arithmetic weights. 

Our proof of Theorem ~\ref{mainthm} uses the $\delta$-method, which in its current form has its genesis in the work of Duke, Friedlander and Iwaniec ~\cite{DFI} on the subconvexity problem for $GL_2$ $L$-functions. In our present work, it is more convenient to use a variant of this method developed by Heath-Brown ~\cite{HB}. 
\begin{remark}
One could also consider estimating $N(\lambda;X)$ by parametrising solutions to $F=0$. To illustrate this, let $F = x_1x_2-x_3^2-x_4^2$. Solutions to $F=0$ in $\mathbf{P}^3$ can be parametrised as $[y_2^2+y_3^2:y_1^2:y_1y_2:y_1y_3]$, with $[y_1:y_2:y_3]\in \mathbf{P}^2$. Thus studying $N(\lambda;X)$ reduces to studying sums of the form
\begin{equation*}
\begin{split}
\sum_{g \leq X}\hspace{0.2cm} \sum_{\substack{(y_1,y_2^2+y_3^2)=g \\ y_2^2+y_3^2 \leq gX \\ y_1^2 \leq gX \\ (y_1,y_2,y_3)=1}}\lambda(y_2^2+y_3^2).
\end{split}
\end{equation*}
The innermost sum can potentially be analysed by the methods developed in ~\cite{TT}, although the additional GCD condition makes it a challenging prospect.
\end{remark}
The methods used in this article appear to extend  to cover the case when $f$ is not holomorphic. In this case, we have the bound $$\lambda(n) \ll_{\ve,f} n^{\frac{7}{64}+\ve}$$ due to Kim and Sarnak ~\cite{KS}, but this does not affect the analysis significantly. With more effort, one could also establish a similar result for forms with arbitrary level and central character.

Finally, if $f$ is not a cusp form, we will have to account for the appearance of a main term, but the analysis of the error terms will remain unchanged. Although we omit the details, the proof of Theorem ~\ref{mainthm} can be suitably modified to give an asymptotic formula for $N(a;X)$ when $a(n) = \tau(n)$ or $r(n)$.

We end our introduction by highlighting some of the key ideas in the proof of Theorem ~\ref{mainthm}. As is typical when applying the $\delta$-method, an application of Poisson summation leads us to study sums that are essentially of the form
\begin{equation}\label{eq:intro}
X\sum_{\substack{\bs{c}' \in \mathbf{Z}^3 \\ |\bs{c}'| \ll X^{\ve}}}\sum_{q \ll X}q^{-\frac{3}{2}}\sum_{n \ll X} \lambda(n)T(A_1n^2,F^{-1}(0,\bs{c}');q)I_q(n,\bs{c}').
\end{equation}
Here $I_q(n,\bs{c}')$ is an exponential integral, $F^{-1}$ is the quadratic form dual to $F$, and $T(m,n;q)$ is a certain one-dimensional exponential sum of modulus $q$ which, on average, admits square-root cancellation (for fixed $m$, say). The derivatives $\frac{\partial^j}{\partial n^j}I_q(n,\bs{c}')$ depend polynomially on $X/q$, and determining how to control them is one of the main challenges we shall face. 

Using Deligne's bound for $\lambda(n)$ and the bound $I_q(n,\bs{c}') \ll 1$, we see that the sum in ~\eqref{eq:intro} is $O(X^{2+\ve}).$ This will be our starting point, and our objective is to make some saving in the $n$-sum. The analysis differs according to the vanishing of $F^{-1}(0,\bs{c}')$. 

If $F^{-1}(0,\bs{c}') = 0$, then $T(A_1n^2,0;q)$ is essentially a Gauss sum, which can be evaluated explicitly.  The problem then reduces to estimating sums of the form $\sum_{d \mid n}\chi(n)\lambda(n)I_q(n,\bs{c}'),$ for $\chi$ a Dirichlet character with conductor $O_{F}(1)$. One of the main novelties of our work is controlling the Mellin transform of $I_q(n,\bs{c}')$ by means of a stationary phase argument. The above sum can then be estimated via Perron's formula, which leads to requiring a subconvexity estimate for the twisted $L$-function, $L(s,f\otimes \chi)$. 

On the other hand, if $F^{-1}(0,\bs{c}')$ does not vanish, Voronoi's formula works well when $q$ is a small power of $X$. Indeed, if $w$ has support in $[X,2X]$ and its derivatives satisfy the bound $w^{(j)}(x) \ll_j x^{-j}$, Voronoi's identity transforms the sum $\textstyle\sum \lambda(n)e_q(an)w(n)$ to a `short' sum of length about $q^2/X$, when $(a,q)=1$. However, in our current regime, the derivatives of $I_q(n,\bs{c}')$ are too large for small $q$, and we must balance these opposing forces to make a saving in the $n$-sum. When $q$ is large, the derivatives are under control, and we can estimate the $n$-sum using partial summation making use of the classical bound $\sum_{n \leq X}e_q(\alpha n)\lambda(n) \ll X^{1/2}\log X.$ 
 
\begin{ack}
I would like to thank my supervisor, Tim Browning, for suggesting this problem to me, and for his detailed comments on earlier drafts of the paper which have significantly improved its exposition. I would also like to thank the referee for their many helpful comments, corrections and suggestions (especially for suggesting the use of Lemma~\ref{referee} in the proof of Theorem~\ref{mainthm}), that greatly helped me refine my thinking through this paper. Part of this work was done while I was a Program Associate in the Analytic Number Theory Program at the Mathematical Sciences Research Institute, Berkeley, USA, during Spring Semester 2017, which was supported by the National Science Foundation under grant no. DMS-1440140.
\end{ack}
\subsection*{Notation}
We write $4$-tuples $\bs{c} = (c_1,\ldots,c_4)$ as $\mathbf{c} = (c_1,\bs{c'})$, where $\bs{c'} = (c_2,c_3,c_4)$ is  a 3-tuple. Let $S(m,n;q) = \textstyle\sum_{x \pmod{q}}^{*}e_q(mx+n\overline{x})$ denote the standard Kloosterman sum, and let $c_q(m) = S(m,0;q)$ be Ramanujan's sum. For an integer $n$, $v_p(n)$ will denote its valuation at a prime $p$. If $F$ is a non-singular quadratic form, we denote by $F^{-1}$ the form dual to $F$; by $\Delta$ we denote the discriminant of $F$. We use the notation $\bs{1}_{S}$ to denote the indicator function of a set $S$. Given a smooth function $w:\mathbf{R}^n \to \mathbf{R}$, we will denote by $\|w\|_{N,1}$ its $L^1$ Sobolev norm of order $N$. All implicit constants will be allowed to depend on the quadratic form $F$, the cusp form $f$ and the weight function $w$. Any further dependence will be indicated by an appropriate subscript.
\section{Preliminaries}
\subsection{Bessel functions}
We begin by recalling some basic properties of the $J$-Bessel function that we will need for the proof of Theorem~\ref{mainthm}. For $x > 0, \nu \geq 2$, we have
\begin{equation}\label{eq:boundforjbessel}
\left(\frac{x}{1+x}\right)^kJ_{\nu}^{(k)}(x) \ll_{k,\nu} \frac{x^{\nu}}{(1+x)^{\nu+\frac{1}{2}}}.
\end{equation}
In particular, observe that $xJ'_{\nu}(x) \ll_{\nu} 1$, for $x \ll 1$. We will also make use of the recurrence relation
\begin{equation}\label{eq:besselrecurrence}
(x^{\nu}J_{\nu}(x))' = x^{\nu}J_{\nu-1}(x).
\end{equation}
For $\nu \geq 2$ define $$W_{\nu}(x) = \frac{e^{i(\frac{\pi}{2}\nu - \frac{\pi}{4})}}{\Gamma(\nu+\frac{1}{2})}\sqrt{\frac{2}{\pi x}}\int_{0}^{\infty}e^{-y}\left(y(1+\frac{iy}{2x})\right)^{\nu-\frac{1}{2}}\, dy.$$
By~\cite[Page 206]{W44}, we have
\begin{equation}\label{eq:hankeldefinition}
J_{\nu}(x) = e^{ix}W_{\nu}(x) + e^{-ix}\overline{W_{\nu}(x)}.
\end{equation}
Moreover, one can verify that for $a \geq 0$ we have \begin{equation}\label{eq:boundforhankel}x^aW_{\nu}^{(a)}(x) \ll_{a,\nu} x(1+x)^{-\frac{3}{2}},\end{equation}
whenever $x \gg 1$. 

\subsection{Summation formulae}
The following lemma is a standard application of Poisson summation.
\begin{lemma}\label{poisson1}
Let $w(x)$ be a smooth function with compact support. Then
\begin{equation}
\sum_{m \equiv b \mod{q}} w(m) = \frac{1}{q} \sum_{m \in \mathbf{Z}} \widehat{w}\left(\frac{m}{q}\right)e_q(bm),
\end{equation}
and $\widehat{w}$ denotes the Fourier transform of $w$.
\end{lemma}
Next, we state a form of the Voronoi summation formula. For a proof, we refer the reader to ~\cite[Proposition 2.1]{FGKM}.
\begin{lemma}\label{progr}
Let $g(x)$ be a smooth function with compact support, and $\lambda(m)$ be the normalised Fourier coefficients of a cusp form of weight $k$ and full level. We then have
\begin{equation}
\sum_{m \equiv b \mod{q}} \lambda(m)g(m) =  \frac{1}{q} \sum_{d \mid q} \sum_{m=1}^{\infty}\lambda(m)S(b,m;d)\check{g}_d(m),
\end{equation}
where 
\begin{equation}\label{eq:hankeltransform}
\check{g}_d(m) = \frac{2\pi i^{k}}{d} \int_{0}^{\infty} g(x) J_{k-1}\left(\frac{4\pi}{d}\sqrt{xm}\right) \, dx,
\end{equation}
is a Hankel-type transform of $g$.
\end{lemma}
\begin{lemma}\label{trunc}
Let $g \in C^{\infty}_{0}(\mathbf{R})$ be a smooth function with support in $[1/2,2]$. Then for any $l \geq 0$ we have
\begin{equation*}
\int_{0}^{\infty} g(x)J_{k-1}\left(t\sqrt{x}\right) \, dx \ll_l \|g\|_{l,1}t^{-(l+1/2)}. 
\end{equation*}
\end{lemma}
\begin{proof}
Denote the left hand side above by $I(t)$. Although this is a standard argument, we present a proof from ~\cite[Proposition 2.3]{FGKM}. Set $\alpha = t^{-2}$. Making the change of variables $x \to \alpha y^2$ we see that
\begin{equation*}
I(t) = 2\alpha \int_{0}^{\infty} g(\alpha y^2) y J_{k-1}(y) \, dy.  
\end{equation*}
By~\eqref{eq:besselrecurrence} and by repeated integration by parts we have
\begin{equation*}
I(t) = 2\alpha \int_{0}^{\infty} \left\{\sum_{0\leq v \leq l} \xi_{v,l} (\alpha y^2)^v g^{(v)}(\alpha y^2)\right\} \frac{J_{k-1+l}(y)}{y^{l-1}} \, dy,
\end{equation*}
for some constants $\xi_{v,l}$. Since $J_{k-1+l}(y) \ll (1+y)^{-1/2}$ and $y \asymp \alpha^{-\frac{1}{2}}$, we see that
\begin{equation*}
I(t) \ll_l \|g\|_{l,1}t^{-(l+1/2)}.
\end{equation*}
This completes the proof.
\end{proof}

\subsection{Some facts about $L$-functions}
In this section, we collect some standard facts about $L$-functions; we refer the reader to~\cite[Chapter 5]{IK} for a more comprehensive account of the theory. Let $f$ be a Hecke eigenform of weight $k$ and full level with normalised Fourier coefficients $\lambda(n)$ as before. Let $\chi$ be a primitive Dirichlet character with conductor $D$. For $\sigma > 1$ let 
\begin{equation*}\label{eq:lsfotimeschi}
L(s,f \otimes \chi) = \sum_{n=1}^{\infty} \frac{\chi(n)\lambda(n)}{n^s}.
\end{equation*}
Then $L(s,f \otimes \chi)$ has analytic continuation to the entire complex plane, satisfies a functional equation, and has the Euler product
\begin{equation}\label{eq:eulerproduct}
L(s, f\otimes \chi) = \prod_{p} \left(1-\frac{\chi(p)\lambda(p)}{p^s} + \frac{\chi^2(p)}{p^{2s}}\right)^{-1}
\end{equation} 
for $\sigma > 1$. Applying the Phragm\'en-Lindel\"of principle in the region $\frac{1}{2} \leq \sigma \leq 1$ to $L(s,f\otimes \chi^{})$, we get that
\begin{equation}\label{eq:phraglind}
L(s,f\otimes \chi) \ll_{\ve,f} (D(1+|t|))^{1-\sigma+\ve},
\end{equation} 
for any $\ve > 0$. When $\sigma = \frac{1}{2}$, we can improve on ~\eqref{eq:phraglind}. We record the following subconvexity bounds for $L(s,f\otimes \chi)$. It follows from ~\cite{BMN} that there exists $A > 0$ such that for all $\ve > 0$ we have
\begin{equation}\label{eq:weyl}
L(s,f\otimes \chi) \ll_{\ve,f} D^A(1+|t|)^{\frac{1}{3}+\ve}.
\end{equation}
Although they are not used here, `hybrid' subconvexity bounds for $L(s,f \otimes \chi)$ are also known, thanks to the work of Blomer and Harcos ~\cite{BH}, and Munshi ~\cite{M}: there exists $\delta >0$ such that for all $\ve >0$ we have
\begin{equation*}
L(s,f \otimes \chi) \ll_{\ve,f} (D(1+|t|))^{\frac{1}{2}-\delta+\ve}.
\end{equation*}

\section{Setting up the $\delta$-method}
Let $$\delta(n) = \begin{cases} 1 &\mbox{$n = 0$, } \\
0 &\mbox{otherwise.}\end{cases}$$
By ~\cite[Theorem 1]{HB} there exists a function $h:\mathbf{R}^{+} \times \mathbf{R} \to \mathbf{R}$ such that for any $Q \geq 1$,  
\begin{equation*}
\delta(n) = c_Q Q^{-2} \sum_{q=1}^{\infty} \sideset{}{^{*}}\sum_{a \mod{q}} e_q(an) h\left(\frac{q}{Q},\frac{n}{Q^2}\right),
\end{equation*}
where $c_Q = 1 + O_A(Q^{-A})$. The function $h(x,y)$ vanishes unless \linebreak $x \leq \min(1,2|y|).$ Its derivatives satisfy the bound \begin{equation}\label{eq:derh}\frac{\partial^{{a+b}}}{\partial x^a y^b}h(x,y) \ll_{N,a,b} x^{-1-a-b}\left(\delta(b)x^N+\min\left(1,\frac{x}{|y|}\right)^N\right),\end{equation} and $h(x,y)$ resembles the $\delta$-distribution in the following sense:
$$\int_{\mathbf{R}} h(x,y)f(y) \, dy = f(0) + O_{f,N}(x^N).$$
Using the $\delta$-symbol to detect the equation $F(\bs{x}) = 0$ we see that
\begin{equation}\label{eq:1}
N(\lambda,X) = c_QQ^{-2}\sum_{q=1}^{\infty}\sideset{}{^{*}}\sum_{a \mod{q}}\sum_{\bs{x} \in \mathbf{Z}^4}\lambda(x_1)e_q(aF(\bs{x}))w\left(\frac{\bs{x}}{X}\right)h\left(\frac{q}{Q},\frac{F(\bs{x})}{Q^2}\right).
\end{equation}
We will take $Q=X$ in our application of the $\delta$-method, since $F(\bs{x})$ is typically of size $X^2$ when $\bs{x}$ is of size $X$. For the rest of this note, we fix the quadratic form to be $$F(\bs{x}) =  A_1x_1^2 + \ldots + A_4x_4^2,$$ 
for non-zero integers $A_1, \ldots, A_4$.

\subsection{Applying the Poisson summation formula}
Letting $Q=X$ and breaking up the sum in ~\eqref{eq:1} into residue classes modulo $q$ we get
\begin{equation*}
\begin{split}
N(\lambda;X) = c_QQ^{-2}\sum_{q=1}^{\infty}\,&\sideset{}{^*}\sum_{a \mod{q}}\sum_{\bs{b} \mod{q}}e_q(aF(\bs{b})) \\ &\times \sum_{\bs{x} \equiv \bs{b} \mod{q}}\lambda(x_1)w\left(\frac{\bs{x}}{X}\right)h\left(\frac{q}{Q},\frac{F(\bs{x})}{Q^2}\right).
\end{split}
\end{equation*}
Applying Lemma ~\ref{poisson1} in the $x_2,x_3$ and $x_4$ variables we get that
\begin{equation}\label{eq:prevoronoi}
\begin{split}
N(\lambda;X) = c_QX\sum_{q=1}^{\infty}q^{-3}&\sum_{\bs{c}' \in \mathbf{Z}^3}\hspace{0.2cm}\sideset{}{^*}\sum_{a \mod{q}}\sum_{\bs{b} \mod{q}}e_q(aF(\bs{b})+\bs{b}'.\bs{c}') \\
&\times\sum_{c_1 \equiv b_1 \mod{q}}\lambda(c_1)I_q(\bs{c}),
\end{split}
\end{equation}
where if $r=q/X$,
\begin{equation*}
I_q(\bs{c}) = \int_{\mathbf{R}^3} w(c_1/X,\bs{z})h(r,F(c_1/X,\bs{z}))e_{r}(-\bs{c}'.\bs{z})\, d\bs{z}.
\end{equation*}
By properties of the $h$-function we see that $q \ll X$, or equivalently, $r \ll 1$. 

Set $\bs{u}' = r^{-1}\bs{c}'$, 
\begin{equation*}
F_q(b_1,s) = \sum_{n \equiv b_1 \mod{q}}\frac{\lambda(n)}{n^s},
\end{equation*}
and
\begin{equation}\label{eq:iqcs}
I_q(\bs{c}',s) = \int_{\mathbf{R}^+\times \mathbf{R}^3} w(\bs{x})h(r,F(\bs{x}))e(-\bs{u}'.\bs{x}')x_1^{s-1}\, d\bs{x}.
\end{equation}
For $\frac{1}{2} \leq \sigma \leq 2$, integrating by parts we see that
\begin{equation}\label{eq:iqhparts}
\begin{split}
I_q(\bs{c}',s) &\ll |s|^{-N} \left|\int \frac{\partial^N}{\partial x_1^N}\left\{w(\bs{x})h(r,F(\bs{x}))\right\}x_1^{s+N-1}e(-\bs{u}'.\bs{x}')d\bs{x}\right| \\
&\ll_N r^{-1-N}|s|^{-N},
\end{split}
\end{equation}
by ~\eqref{eq:derh}. For $\sigma > 1$, we have $F_q(b_1,s) \ll 1$, as the Dirichlet series converges absolutely in this region. By the Mellin inversion theorem, we therefore have
\begin{equation*}
\begin{split}
N(\lambda;X) = c_QX\sum_{q \ll X}q^{-3}&\sum_{\bs{c}' \in \mathbf{Z}^3}\hspace{0.2cm}\sideset{}{^*}\sum_{a \mod{q}}\sum_{\bs{b} \mod{q}}e_q(aF(\bs{b})+\bs{b}'.\bs{c}') \\ &\times\frac{1}{2\pi i}\int_{(\sigma)}X^s F_q(b_1,s) I_q(\bs{c}',s)\, ds,
\end{split}
\end{equation*}
whenever $\sigma > 1$.

We end this section by recording an alternate expression for $N(\lambda;X)$. Applying Lemma ~\ref{progr} to the $c_1$ variable in ~\eqref{eq:prevoronoi} we see that
\begin{equation}\label{eq:postvoronoi}
N(\lambda;X) = c_QX^{2}\sum_{q\ll X}q^{-4}\sum_{\substack{\bs{c}  \in \mathbf{Z}^4\\c_1 \geq 1}}\lambda(c_1) \sum_{d \mid q}S_{d,q}(\bs{c})I_{d,q}(\bs{c}),
\end{equation}
where
\begin{equation}\label{eq:s2}
\begin{split}
S_{d,q}(\bs{c}) &= \sideset{}{^{*}}\sum_{a \mod{q}}\sum_{\bs{b} \mod{q}} e_q(aF(\bs{b})+\bs{b'}.\bs{c}')S(b_1,c_1;d)
\end{split}
\end{equation}
and
\begin{equation}\label{eq:i2}
\begin{split}
I_{d,q}(\bs{c}) &= \frac{2\pi i^{k}}{d} \int_{\mathbf{R}^+\times\mathbf{R}^3} w\left(\bs{x}\right) h\left(r,F(\bs{x})\right) J_{k-1}\left(\frac{4\pi}{d}\sqrt{c_1Xx_1}\right) e\left(-\bs{u}'.\bs{x'}\right) \, d\bs{x}.
\end{split}
\end{equation}

\section{Integral estimates}
\subsection{First steps}
Let
\begin{equation}\label{eq:testtest}
w_0(x) = \begin{cases} \exp(-(1-x^2)^{-1}), &\mbox{$|x| < 1$} \\
0 &\mbox{$|x| \geq 1$},
\end{cases}
\end{equation}
be a smooth function with compact support and let $$\gamma(x) = w_0\left(\frac{x}{100\max_{i=1,2,3,4}|A_i|}\right).$$ Then $\gamma(F(\bs{x})) \gg 1$ whenever $\bs{x} \in \supp(w)$. Recall that $r=q/X$ and let \begin{equation}\label{eq:g}g(r,y) = h(r,y)\gamma(y).\end{equation} 
Then for each $r$, the function $g(r,y)$ has compact support, and by~\cite[Lemma 17]{HB} we have the following bound for its Fourier transform,
\begin{equation}\label{eq:prt}
\begin{split}
p_{r}(t) &= p(t) = \int_{\mathbf{R}} g(r,y)e(-ty) \, dy \ll_j (r|t|)^{-j}.
\end{split}
\end{equation}
\begin{remark}
The above bound shows that $p(t)$ has polynomial decay unless $|t| \ll r^{-1-o(1)}.$
\end{remark}

We also record a certain dissection argument due to Heath-Brown~\cite[Lemma~2]{HB}. Let $w_0$ be as in~\eqref{eq:testtest}, and let $c_0 = \textstyle\int_{\mathbf{R}} w_0(x) \,dx.$ For $\bs{u},\bs{v} \in \mathbf{R}^3$ define
\begin{equation*}\label{eq:oct9}
w_{\delta}(x_1, \bs{u},\bs{v}) = c_0^{-3}w_0^{(3)}(\bs{u})w(x_1,\delta \bs{u} + \bs{v}),
\end{equation*}
where $$w_0^{(3)}(\bs{u}) = \prod_{i=1}^3 w_0(u_i).$$ Then 
\begin{equation}\label{eq:dissection}
\int_{\mathbf{R}^3} w_{\delta}\left(x_1,\frac{\bs{x}'-\bs{y}'}{\delta},\bs{y}'\right) \, d\bs{y}' = \delta^3 w(\bs{x}). 
\end{equation}
Finally, we remind the reader that the test function $w$ is supported in $[1/2,2]^4$, a fact we will repeatedly make use of. It is crucial to our arguments that $w$ is supported away from the origin.
\subsection{Estimates for $I_q(\bs{c})$}
Recall that
\begin{equation*}
I_q(\bs{c}) = \int_{\mathbf{R}^3} w(c_1/X,\bs{z})h(r,F(c_1/X,\bs{z}))e(-\bs{u}'.\bs{z})\, d\bs{z}.
\end{equation*}
We have the following estimates.
\begin{lemma}\label{iqctrivial}
$I_q(\bs{c}) \ll 1.$
\end{lemma}
\begin{proof}
This follows from~\cite[Lemma 15]{HB}. 
\end{proof}
\begin{lemma}\label{lemmatruncatec'}
Let $N\geq 0$ and suppose that $\bs{c}' \neq \bs{0}$. Then
\begin{equation*}
I_q(\bs{c}) \ll_{N} \frac{X}{q}|\bs{c}'|^{-N}.
\end{equation*}
\end{lemma}
\begin{proof}
This follows from~\cite[Lemma 19]{HB}.
\end{proof}
As a consequence of Lemma~\ref{lemmatruncatec'}, we find that $I_q(\bs{c}) \ll_A X^{-A}$ if $|\bs{c}'| \gg X^{\ve}$. It remains to examine the behaviour of $I_q(\bs{c})$ when $|\bs{c}'| \ll X^{\ve}$.
\begin{lemma}\label{lemmastationaryphase}
Let $\ve > 0$. Suppose that $1 \leq |\bs{c}'| \ll X^{\ve}$. Then, for $j = 0,1$ we have
\begin{equation*}
\frac{\partial^j}{\partial c_1^j}I_q(\bs{c}) \ll_{\ve} (r^{-1}|\bs{u}'|)^{\ve}\left(r^{-j}\left(\frac{c_1}{X^2}\right)^{j} + \frac{j}{X}\right)|\bs{u}'|^{-\frac{1}{2}}.
\end{equation*}
\end{lemma}
\begin{proof}
Since $r \ll 1$, we have $|\bs{u}'| \gg 1$ under the hypotheses of the lemma. By~\eqref{eq:prt} we have
\begin{equation*}
I_q(\bs{c}) = \int_{\mathbf{R}} p(t) \int _{\mathbf{R}^3}\widetilde{w}(c_1/X,\bs{z})e(tF(c_1/X,\bs{z})-\bs{u}'.\bs{z})\, d\bs{z} \, dt,
\end{equation*}
where $$\widetilde{w}(c_1/X,\bs{z}) = \frac{w(c_1/X,\bs{z})}{\gamma(F(c_1/X,\bs{z}))}.$$
For $j \in \left\{0,1\right\}$, 
\begin{equation*}
\begin{split}
\frac{\partial^j}{\partial c_1^j}I_q(\bs{c}) &= \left(\frac{4\pi iA_1c_1}{X^2}\right)^j\int t^j p(t) \int \widetilde{w}(c_1/X,\bs{z})e(tF(c_1/X,\bs{z})-\bs{u}'.\bs{z}) \, d\bs{z} \, dt \\
&\quad \quad + \frac{j}{X}\int p(t)\int \widetilde{w}^{(1,\bs{0})}(c_1/X,\bs{z})e(tF(c_1/X,\bs{z})-\bs{u}'.\bs{z})\,d\bs{z} \, dt.
\end{split}
\end{equation*}
Denote the integrals over $\bs{z}$ by $I_1(t)$ and $I_2(t)$ respectively. Using~\cite[Lemma 3.1]{HBP} we have the following bounds for $I_k(t)$: 
\begin{equation*}
I_k(t) \ll \min\left(1,|t|^{-\frac{3}{2}}\right),
\end{equation*}
and if $|\bs{u}'| \gg |t|$ then
\begin{equation*}
I_k(t) \ll_{N} |\bs{u}'|^{-N},
\end{equation*}
for $k = 1,2.$ By~\eqref{eq:prt} we have the bounds
\begin{equation*}
\int_{|t| \ll |\bs{u}'|} |t|^j|p(t)| \ll |\bs{u}'|^{1+j},
\end{equation*}
and 
\begin{equation*}
\int_{|t| \gg |\bs{u}'|} |p(t)||t|^{j-\frac{3}{2}} \,dt \ll_j r^{-j}|\bs{u}'|^{-\frac{1}{2}}. 
\end{equation*}
As a result,
\begin{equation}\label{eq:iqpartial1}
\frac{\partial^j}{\partial c_1^j}I_q(\bs{c}) \ll_N \left(\frac{c_1}{X^2}\right)^j\left(|\bs{u}'|^{1+j-N}+r^{-j}|\bs{u}'|^{-\frac{1}{2}}\right) + \frac{j}{X}\left(|\bs{u}'|^{1-N} + |\bs{u}'|^{-\frac{1}{2}}\right).
\end{equation}
We will see that this is satisfactory for the lemma unless $|\bs{c}'|$ is essentially $O(1)$. If this is the case, we proceed as follows. 

For $j  \in \{0,1\}$ we have
\begin{equation*}
\begin{split}
\frac{\partial^j I_q(\bs{c})}{\partial c_1^j} &= \left(\frac{2A_1c_1}{X^2}\right)^j\int_{\mathbf{R}^3}w(c_1/X,\bs{z})h^{(0,j)}(r,F(c_1/X,\bs{z}))e(-\bs{u}'.\bs{z}') \, d\bs{z} \, \\ 
&\quad \quad + \frac{j}{X} \int w^{(j,\bs{0})}(c_1/X,\bs{z})h(r,F(c_1/X,\bs{z})e(-\bs{u}'.\bs{z})\, d\bs{z} \\
\end{split}
\end{equation*}
Applying~\eqref{eq:derh} with $N=2+j$ in the first integral, and with $N=2$ in the second, we have 
\begin{equation*}
\begin{split}
\frac{\partial^j I_q(\bs{c})}{\partial c_1^j} & \ll \left(\frac{c_1}{X^2}\right)^j \int |w(c_1/X,\bs{z})|r^{-1-j}\left\{\delta(j)r^{2+j}+\min \left(1,\frac{r^{2+j}}{F(c_1/X,\bs{z})^{2+j}}\right) \right\}\, d\bs{z}\,   \\
&\quad \quad + \frac{j}{X} \int |w^{(j,\bs{0})}(c_1/X,\bs{z})|r^{-1}\left\{r^2+\min \left(1,\frac{r^2}{F(c_1/X,\bs{z})^2}\right) \right\}\, d\bs{z}. \\
\end{split}
\end{equation*}
Observe that the measure of the set of $\bs{z}$ for which $|F(\frac{c_1}{X},\bs{z})| \ll \nu$ is $O(\nu)$. Moreover, there exists $C > 0$ such that $F(c_1/X,\bs{z}) \ll C$ in the support of $w$. As a result, we find that
\begin{equation*}
\begin{split}
I_q(\bs{c}) &\ll r+r^{-1} \int_{|F(c_1/X,\bs{z})| \leq r} |w(c_1/X,\bs{z})| \, d\bs{z}+ r\int_{|F(c_1/X,\bs{z})|\geq r}\frac{|w(c_1/X,\bs{z})|}{F(c_1/X,\bs{z})^2} \, d\bs{z}\\
&\ll 1 + r \sum_{u=0}^{\log_2 C/r}\int_{2^{u}r \leq |F(c_1/X,\bs{z})| \leq 2^{u+1}r} \frac{|w(c_1/X,\bs{z})|}{4^ur^2}\, d\bs{z} \\
&\ll 1 + 2r\sum_{u=0}^{\log_2 C/r} \frac{1}{2^u r} \ll 1.
\end{split}
\end{equation*} Similarly,
\begin{equation*}
\begin{split}
\frac{\partial I_q(\bs{c})}{\partial c_1} &\ll  \left(\frac{r^{-2}c_1}{X^2}\right)\left\{\int_{|F(c_1/X,\bs{z})| \leq r} |w(c_1/X,\bs{z})| \, d\bs{z} + r^3\int_{|F(c_1/X,\bs{z})|\geq r}\frac{|w(c_1/X,\bs{z})|}{|F(c_1/X,\bs{z})|^3} \, d\bs{z}\right\} \\
&+ \frac{r^{-1}}{X}\left\{ \int_{|F(c_1/X,\bs{z})| \leq r} |w^{(1,\bs{0})}(c_1/X,\bs{z})| \, d\bs{z} + r^2\int_{|F(c_1/X,\bs{z})|\geq r}\frac{|w^{(1,\bs{0})}(c_1/X,\bs{z})|}{F(c_1/X,\bs{z})^2} \, d\bs{z} \right\}\\
&\ll r^{-1}\left(\frac{c_1}{X^2}\right) + \frac{1}{X}.
\end{split}
\end{equation*}
To sum up, for $j =0,1$ we have shown
\begin{equation}\label{eq:iqcj}\frac{\partial^j}{\partial c_1^j}I_q(\bs{c}) \ll r^{-j}\left(\frac{c_1}{X^2}\right)^{j} + \frac{j}{X}.
\end{equation} 

We are now in place to finish the proof of the lemma. Suppose first that $|\bs{u}'| \ll r^{-2\ve/3}$, then $|\bs{u}'|^{\frac{1}{2}-\ve} \ll r^{-\ve}$. In this case, $$\frac{\partial^j}{\partial c_1^j}I_q(\bs{c}) \ll r^{-j}\left(\frac{2c_1}{X^2}\right)^{j} + \frac{j}{X} \ll (r^{-1}|\bs{u}'|)^{\ve}\left(r^{-j}\left(\frac{2c_1}{X^2}\right)^{j} + \frac{j}{X}\right)|\bs{u}'|^{-\frac{1}{2}}.$$ Suppose next that $|\bs{u}'| \gg r^{-\frac{2\ve}{3}}$ then choosing $N$ large enough in~\eqref{eq:iqpartial1} we get $$\frac{\partial^j}{\partial c_1^j}I_q(\bs{c}) \ll \left(r^{-j}\left(\frac{c_1}{X^2}\right)^{j} + \frac{j}{X}\right)|\bs{u}'|^{-\frac{1}{2}}.$$ This completes the proof of the lemma.\end{proof}
\begin{remark}
The reader should compare the preceding result to~\cite[Lemma 22]{HB}.
\end{remark}
\subsection{Estimates for $I_q(\bs{c}',s)$ and $I_{d,q}(\bs{c})$}
In this section we will denote the complex variable $s = \sigma + it$, where $\sigma$ and $t$ are real. Recall the integrals in~\eqref{eq:iqcs},~\eqref{eq:i2}, 
\begin{equation*}
I_q(\bs{c}',s) = \int_{\mathbf{R}^+\times \mathbf{R}^3} w(\bs{x})h(r,F(\bs{x}))e(-\bs{u}'.\bs{x}')x_1^{s-1}\, d\bs{x}
\end{equation*}  
and 
\begin{equation*}
\begin{split}
I_{d,q}(\bs{c}) &= \frac{2\pi i^{k}}{d} \int_{\mathbf{R}^+\times\mathbf{R}^3} w\left(\bs{x}\right) h\left(r,F(\bs{x})\right) J_{k-1}\left(\frac{4\pi}{d}\sqrt{c_1Xx_1}\right) e\left(-\bs{u}'.\bs{x'}\right) \, d\bs{x}.
\end{split}
\end{equation*}
\subsubsection{First estimates}
The following `trivial' bounds follow from~\cite[Lemma 15]{HB} and the bound~\eqref{eq:boundforjbessel}. Our task for the rest of the section will be to improve upon them. 
\begin{lemma}\label{iqcstrivialbound}
Let $\frac{1}{2} \leq \sigma \leq 2$. We have $$I_q(\bs{c}',s) \ll 1$$
and
$$I_{d,q}(\bs{c}) \ll \frac{1}{d}\left(1+\frac{\sqrt{c_1X}}{d}\right)^{-\frac{1}{2}}.$$  
\end{lemma}
Integrating by parts, next we will give estimates for $I_q(\bs{c}',s)$ and $I_{d,q}(\bs{c})$ in the spirit of Lemma~\ref{lemmatruncatec'}. 
\begin{lemma}\label{lemmaczero}
Let $\frac{1}{2} \leq \sigma \leq 2$. We have
\begin{equation*}
I_q(\bs{0},s) = I_q(s) \ll_A \min\left\{1,|s|^{-A}\right\}.
\end{equation*}
\end{lemma}
\begin{proof}
We write
\begin{equation*}
\begin{split}
I_q(s) &= \int_{\mathbf{R}} p(u) \int_{\mathbf{R}^{+}\times \mathbf{R}^3} \widetilde{w}(\bs{x})e(\Phi(u,\bs{x}))\, d\bs{x} \, du, 
\end{split}
\end{equation*}
with $\Phi(u,\bs{x}) = uF(\bs{x})+\frac{(s-1) \log x_1}{2\pi i}.$ Since $F$ is diagonal and $w$ is supported in the box $[1/2,2]^4$, we see that $|\nabla F(\bs{x})| \gg 1$ in the support of $w$; as a result we find that $|\nabla \Phi| \gg |u|$. Furthermore, if $|u| \ll |s|$ we see that $|\nabla \Phi| \gg |s|$. Therefore, we have by ~\cite[Lemma 10]{HB} and ~\eqref{eq:prt} that for $A \geq 0$
\begin{equation*}
\begin{split}
I_q(s) &\ll_A |s|^{-A}\int_{|u| \ll |s|}|p(u)| \, du + \int_{|u| \gg |s|}|u|^{-A}|p(u)| \, du \\
&\ll_A |s|^{1-A},
\end{split}
\end{equation*}
since for all $j=j_1+\ldots + j_4 \geq 2$ we have $$\left|\frac{\partial^{j}\Phi(u,\bs{x})}{\partial^{j_1} x_1\ldots \partial^{j_4}x_4}\right| \ll_j |s|.$$ This completes the proof.
\end{proof}

By Fourier inversion we may write
\begin{equation}\label{eq:star}
I_q(\bs{c}',s) = \int_{\mathbf{R}} p(\alpha) \int_{\mathbf{R}^+\times \mathbf{R^3}} \widetilde{w}(\bs{x})e(\alpha F(\bs{x}) - \bs{u}'.\bs{x}')x_1^{s-1} \, d\bs{x} \, d\alpha, 
\end{equation}
where \begin{equation}\label{eq:widew}
\widetilde{w}(\bs{x}) = \frac{w(\bs{x})}{\gamma(F(\bs{x}))}.
\end{equation}
Let \begin{equation}\label{eq:psidefinition} \Psi(\bs{x}') = \alpha F(0,\bs{x}') - \bs{u}'.\bs{x}'.\end{equation} Therefore,
\begin{equation*}
I_q(\bs{c}',s) = \int_{\mathbf{R}} p(\alpha) \int_{\mathbf{R}^+} x_1^{s-1}e(\alpha A_1x_1^2)\int_{\mathbf{R}^3} \widetilde{w}(\bs{x})e(\Psi(\bs{x}')) \, d\bs{x}' \, dx_1 \, d\alpha.
\end{equation*}
If $|\alpha| \ll |\bs{u}'|$, then $|\nabla \Psi(\bs{x}')| \gg |\bs{u}'|,$ and as a result, the integral over $\bs{x}'$ is $O(|\bs{u}'|^{-N})$ by~\cite[Lemma 10]{HB}. Since the integral over $\bs{x}'$ is trivially $O(1)$, we have the bound
\begin{equation*}
I_q(\bs{c}',s) \ll_N |\bs{u}'|^{-N}\int_{|\alpha|\ll |\bs{u}'|}|p(\alpha)| \, d\alpha + \int_{|\alpha| \gg |\bs{u}'|}|p(\alpha)| \, d\alpha.
\end{equation*}
Therefore, by~\eqref{eq:prt} we have established the following result.
\begin{lemma}\label{lemmaperronlarge}
Suppose that $\bs{c}' \neq \bs{0}$ and $\frac{1}{2} \leq \sigma \leq 2$. Then
\begin{equation*}
I_q(\bs{c}', s) \ll \min\left\{1, r^{-1}|\bs{c}'|^{-N}\right\}.
\end{equation*}
As a result, $I_q(\bs{c}',s) \ll_A X^{-A}$ unless $|\bs{c}'| \ll X^{\ve}$.
\end{lemma}
 
Finally, we turn to $I_{d,q}(\bs{c})$. 
\begin{lemma}\label{truncatem1}
For all $N \geq 0$ we have 
\begin{equation*}
I_{d,q}(\bs{c}) \ll_{N} (dr)^{-1}\left(1+\frac{\sqrt{c_1X}}{d}\right)^{-\frac{1}{2}}\min\left\{|\bs{c}'|^{-N}, \left(\frac{X}{q}\frac{d}{\sqrt{c_1X}}\right)^{N-1}\right\}
\end{equation*}
As a result, $I_{d,q}(\bs{c}) \ll_A X^{-A}$ whenever $c_1 \gg X^{1+\ve}/(q/d)^2$, or $|\bs{c}'| \gg X^{\ve}$.  
\end{lemma}
\begin{proof}
The first estimate follows by applying the argument in the proof of~\cite[Lemma 19]{HB} to the test function $w(\bs{x})J_{k-1}\left(4\pi\frac{\sqrt{c_1Xx_1}}{d}\right)$ in~\eqref{eq:i2}, and integrating by parts in the $\bs{x}'$ variable. As a result, by using the bound~\eqref{eq:boundforjbessel}, we get for any integer $N \geq 0$ that $$I_{d,q}(\bs{c}) \ll (dr)^{-1}\left(1+\frac{\sqrt{c_1X}}{d}\right)^{-\frac{1}{2}}|\bs{c}'|^{-N}.$$

By~\eqref{eq:g} we may write
\begin{equation*}
I_{d,q}(\bs{c}) = \frac{2 \pi i^k}{d} \int_{\mathbf{R}^+ \times \mathbf{R}^3}\widetilde{w}(\bs{x})g(r,F(\bs{x}))J_{k-1}\left(4\pi\frac{\sqrt{c_1Xx_1}}{d}\right)e(-\bs{u}'.\bs{x}')\, d\bs{x}.
\end{equation*}
Applying Lemma~\ref{trunc} with $$\psi(x_1) = \psi_{\bs{x}'}(x_1) = \tilde{w}(\bs{x})g(r,F(\bs{x}))e(-\bs{u}'.\bs{x}'),$$ treated as a function in the variable $x_1$, and $t = 4\pi(c_1X)^{\frac{1}{2}}/d$ we obtain the bound 
\begin{equation*}
I_{d,q}(\bs{c}) \ll_N \frac{1}{d}\|\psi\|_{N,1}\left(4\pi\frac{\sqrt{c_1X}}{d}\right)^{-N-\frac{1}{2}}.
\end{equation*}
Since $\frac{\partial^n}{\partial y^n}h(r,y) \ll r^{-1-n}$, we have the bound $\|\psi\|_{N,1} \ll_N r^{-1-N}$, which gives us 
$$I_{d,q}(\bs{c}) \ll_{N} (dr)^{-1}\left(1+\frac{\sqrt{c_1X}}{d}\right)^{-\frac{1}{2}}\left(\frac{X}{q}\frac{d}{\sqrt{c_1X}}\right)^{N}.$$ 
The result now follows.
\end{proof}

\subsubsection{A stationary phase argument}
Our goal in this section will be to establish the following results.
\begin{lemma}\label{csmall}
Let $\ve >0$ and suppose that $\bs{c}'\neq \bs{0}$. Let $0 < \sigma \leq 2$ we have
$$I_q(\bs{c}',s) \ll_{\ve} \frac{1}{|s|}X^{\ve}.$$
\end{lemma}

\begin{lemma}\label{stphasehb}
Let $\ve>0$ and suppose that $\bs{c}' \neq \bs{0}$. Then 
\begin{equation*}
I_{d,q}(\bs{c}) \ll_{\ve} \frac{1}{d}\left(1+\frac{\sqrt{c_1X}}{d}\right)^{-\frac{1}{2}}\left(\frac{q}{X}\right)X^{\ve}.
\end{equation*}
\end{lemma}
Integrating $I_q(\bs{c}',s)$ by parts we get
\begin{equation}\label{eq:june2}
\begin{split}
I_q(\bs{c}',s) &= -\frac{1}{s}\int \frac{\partial \widetilde{w}(\bs{x})}{\partial x_1}g(r,F(\bs{x}))e(-\bs{u}'.\bs{x}')x_1^s\, d\bs{x}  \\ 
&\quad \quad - \frac{2A_1}{s} \int \widetilde{w}(\bs{x})g^{(0,1)}(r,F(\bs{x}))e(-\bs{u}'.\bs{x}')x_1^{s+1} \, d\bs{x}\\
&= -\frac{1}{s}(I_1+I_2),
\end{split}
\end{equation}
say. On applying Lemma~\ref{iqcstrivialbound} to the test function $\partial \widetilde{w}(\bs{x})/\partial x_1$ we see that \begin{equation}\label{eq:i1bound}I_1 \ll 1\end{equation} and arguing as in~\eqref{eq:iqcj} we see that 
\begin{equation}\label{eq:i2bound}
I_2 \ll r^{-1}.
\end{equation}
Hence it suffices to focus on $I_2$: our task is then to remove the factor $r^{-1}$ in the bound for $I_2$, at the cost of an additional factor of $X^{\ve}$. 

%Next, we \begin{remark}\label{remark42} If $c_1 \ll d^2/X$, then Lemma~\ref{stphasehb} follows from~\cite[Lemma 22]{HB},
%simply by repeating Heath-Brown's argument with the weight function $$w(\bs{x})J_{k-1}\left(4\pi\frac{\sqrt{c_1Xx_1}}{d}\right).$$ To see this, observe that when $c_1 \ll d^2/X$, we have $\frac{\sqrt{c_1Xx_1}}{d} \ll 1$. In this range, the derivatives of the Bessel function are well-behaved by~\eqref{eq:boundforjbessel}. However, this approach no longer works when $c_1 \gg d^2/X$. In this range, we will use the integral representation given by~\eqref{eq:hankeldefinition} to write the Bessel function in terms of exponential functions and smooth functions with rapid decay (which is guaranteed by~\eqref{eq:boundforhankel}). We will then use a more careful stationary phase analysis to prove Lemma~\ref{stphasehb}. 
%\end{remark}

By~\eqref{eq:june2} and~\eqref{eq:prt} we have 
\begin{equation*}\label{eq:postparts}
\begin{split}
I_2 &= 4\pi i A_1\int_{\mathbf{R}} \alpha p(\alpha) \int_{\mathbf{R}^+\times\mathbf{R}^3} \widetilde{w}(\bs{x})x_1^{\sigma+1} e(\alpha A_1x_1^2 + \tfrac{t}{2\pi}\log x_1 + \Psi(\bs{x}')) \, d\bs{x} \, d\alpha, \\
\end{split}
\end{equation*}
where $\Psi(\bs{x}')$ is as in~\eqref{eq:psidefinition}. Similarly, applying~\eqref{eq:prt} to~\eqref{eq:i2} we get 
\begin{equation}\label{eq:11oct1}
I_{d,q}(\bs{c}) = \frac{2\pi i^k}{d}\int_{\mathbf{R}} p(\alpha) \int_{\mathbf{R}^+\times\mathbf{R}^3} \widetilde{w}(\bs{x})J_{k-1}\left(4\pi\frac{\sqrt{c_1Xx_1}}{d}\right)e(\alpha A_1x_1^2+\Psi(\bs{x}'))\, d\bs{x} \, d\alpha. 
\end{equation}
Using~\eqref{eq:hankeldefinition} we have,  
\begin{equation}\label{eq:11oct2}
\begin{split}
I_{d,q}(\bs{c}) &= \frac{2\pi i^k}{d}\int_{\mathbf{R}} p(\alpha) \times  \\ &\quad \quad \int_{\mathbf{R}^{+}\times\mathbf{R}^3} \widetilde{w}(\bs{x})W_{k-1}\left(4\pi\frac{\sqrt{c_1Xx_1}}{d}\right)e(\alpha A_1x_1^2+2\tfrac{\sqrt{c_1Xx_1}}{d}+\Psi(\bs{x}'))\, d\bs{x} \, d\alpha  \\
& + \frac{2\pi i^k}{d}\int_{\mathbf{R}} p(\alpha) \times\\ &\quad \quad \int_{\mathbf{R}^{+}\times\mathbf{R}^3} \widetilde{w}(\bs{x})\overline{W_{k-1}\left(4\pi\frac{\sqrt{c_1Xx_1}}{d}\right)}e(\alpha A_1x_1^2-2\tfrac{\sqrt{c_1Xx_1}}{d}+\Psi(\bs{x}'))\, d\bs{x} \, d\alpha.
\end{split}
\end{equation}
We will treat $I_2$ and $I_{d,q}(\bs{c})$ in a unified way, and our approach is modelled on the proof of~\cite[Lemma 22]{HB}. We will use Heath-Brown's argument in the $\bs{x}'$ variable, and the second derivative estimate for exponential integrals to handle the integral over $x_1$. We will first prove the following version of~\cite[Lemma 20]{HB}.
\begin{lemma}\label{lemmauv}
Let $u:\mathbf{R} \to \mathbf{R}$ be a real-valued function infinitely differentiable on $[1/2,2]$, and let $v:\mathbf{R}^4 \to \mathbf{C}$ be a smooth function with compact support in $[1/2,2]^4$ such that for all $j_k \geq 0$ we have
\begin{align*}
\frac{\partial^{j_2 +j_3 + j_4}v(\bs{x})}{\partial x_2^{j_2}\partial x_3^{j_3}\partial x_4^{j_4}} &\ll_{j_2,j_3,j_4} 1.
\end{align*}
Let $\Psi(\bs{x}')$ be as in~\eqref{eq:psidefinition}. For $j=0,1$ let  
\begin{equation}
\mathcal{I}_j = \int_{\mathbf{R}} \alpha^j p(\alpha) \int_{\mathbf{R}^+\times \mathbf{R}^3} v(\bs{x})e(\alpha A_1x_1^2+u(x_1)+\Psi(\bs{x}')) \, d\bs{x}\, d\alpha.
\end{equation}
Let $R \geq 1$ be a real number and suppose that $|\bs{u}'| \geq R^3$. Then there exists a smooth real valued function $v_1$ with support in $[1/2,2]$, and a real number $\omega$ satisfying $|\bs{u}'| \ll_{F} |\omega| \ll_{F} |\bs{u}'|$ such that for any integer $N \geq 0$ the following holds,
$$\mathcal{I}_j \ll_{N} r^{-1-j}R^{-N}+ R^3r^{\frac{1}{2}-j}\bigg|\int_{1/2}^2 v_1(x)e(A_1\omega x^2+u(x))\, dx\bigg|.$$
\end{lemma}
\begin{proof}
Suppose first that $|\bs{u}'| \geq r^{-1}R$. We argue as in the proof of Lemma~\ref{lemmaperronlarge}. We have
\begin{equation*}
\begin{split}
\mathcal{I}_j &= \int_{\mathbf{R}} \alpha^j p(\alpha) \int_{\mathbf{R}^+} e(\alpha A_1x_1^2 + u(x_1))\int_{\mathbf{R}^3} v(\bs{x})e(\Psi(\bs{x}'))\, d\bs{x}' \, dx_1. 
\end{split}
\end{equation*}
If $|\alpha| \ll |\bs{u}'|,$ then $|\nabla \Psi(\bs{x}')| \gg |\bs{u}'|.$ Therefore, integrating by parts in the $\bs{x}'$-variable using~\cite[Lemma 10]{HB} we get 
\begin{equation*}
\begin{split}
\int e(\alpha A_1x_1^2 + u(x_1))\int v(\bs{x})e(\Psi(\bs{x}'))\, d\bs{x}' \, dx_1 &\ll_M \left(\int_{1/2}^2\|v(x_1,\bs{.})\|_{M,1}\, dx_1 \right)|\bs{u}'|^{-M} \\
&\ll_{M} |\bs{u}'|^{-M},
\end{split}
\end{equation*}
by properties of $v(\bs{x})$. Estimating the integral over $\bs{x}$ trivially when $|\alpha| \gg |\bs{u}'|$, we obtain
\begin{equation*}
\begin{split}
\mathcal{I}_j &\ll_M \int_{|\alpha| \gg |\bs{u}'|} |\alpha|^j |p(\alpha)|\, d\alpha + |\bs{u}'|^{-M}\int_{|\alpha| \ll |\bs{u}'|}|\alpha|^j||p(\alpha)| \, d\alpha\\
\end{split}
\end{equation*}
By~\eqref{eq:prt} we get that 
\begin{equation*}
\begin{split}
\mathcal{I}_j &\ll_{M} r^{-M}|\bs{u}'|^{1+j-M} + |\bs{u}'|^{-M+j} \\ 
&\ll_{N} r^{-1-j}R^{-N}
\end{split}
\end{equation*}
by taking $M = N + j + 1$ and by our assumption that $|\bs{u}'| \geq r^{-1}R$. This completes the proof of the lemma in this case. 

We may therefore suppose that $R^3 \leq |\bs{u}'| \leq r^{-1}R$. Applying~\eqref{eq:dissection} to $v(\bs{x})$ we get that 
\begin{equation*}
\begin{split}
\mathcal{I}_j & =\delta^{-3} \int \alpha^j p(\alpha) \times \\
&\quad \quad \int \int v_{\delta}\left(x_1,\frac{\bs{x}'-\bs{y}}{\delta},\bs{y}\right)e(\alpha A_1x_1^2 + u(x_1) + \Psi(\bs{x}'))\, d\bs{x} \, d\bs{y} \, d\alpha.
\end{split}
\end{equation*}
Let $\bs{x}' = \bs{y} + \delta \bs{z}$. By virtue of $v$ being compactly supported, we see that $|\bs{y}| \ll 1,$ and we arrive at the inequality
\begin{equation*}
\begin{split}
\mathcal{I}_j &\leq \int \int |\alpha^j p(\alpha)| \times \\ &\quad \quad \left|\int v_{\bs{y}}(x_1,\bs{z})e(\alpha A_1x_1^2 + u(x_1) +\Psi(\bs{y}+\delta \bs{z})) \, dx_1 \, d\bs{z}\right| \, d\alpha \, d\bs{y}.
\end{split}
\end{equation*}
with $v_{\bs{y}}(x_1,\bs{z})  = v_{\delta}\left(x_1,\bs{z},\bs{y}\right).$ Henceforth, we will take $\delta = |\bs{u}'|^{-\frac{1}{2}}$. 

Let $\bs{y} = (y_2,y_3,y_4)$ and $\Psi_{\bs{y}}(\bs{z}) = \Psi(\bs{y}+\delta \bs{z}).$ As in the proof of~\cite[Lemma 22]{HB}, we say that a pair $(\bs{y},\alpha)$ is `good', if 
\begin{equation*}
|\nabla \Psi_{\bs{y}}(\bs{0})| = |\bs{u}'|^{-\frac{1}{2}}|2\alpha (A_2y_2,A_3y_3,A_4y_4)- \bs{u}'| \geq R\max \left\{|\alpha|/|\bs{u}'|, 1\right\},
\end{equation*}
and that $(\bs{y},\alpha)$ is `bad' otherwise. If $(\bs{y}, \alpha)$ is `good' then~\cite[Lemma 10]{HB} shows that
\begin{equation*}
\begin{split}
\int \left|\int v_{\bs{y}}(x_1,\bs{z})e(\Psi_{\bs{y}}(\bs{z})) \, d\bs{z}\right|\, dx_1 &\ll_N \left(\int_{1/2}^2\|v_{\bs{y}}(x_1,\bs{.})\|_{N,1}\, dx_1 \right) R^{-N} \\
&\ll_{N} R^{-N},
\end{split}
\end{equation*}
by properties of $v(\bs{x})$. 

Suppose next that $(\bs{y},\alpha)$ is bad. Since $|\bs{u}'|^{-\frac{1}{2}} \ll R^{-\frac{3}{2}},$ observe that $|\bs{y}| \gg_{F} 1$. Moreover, $|\bs{y}| \ll 1$, trivially, and as a result, we see that $|\bs{u}'| \ll |\alpha| \ll |\bs{u}'|.$ Therefore, if $(\bs{y},\alpha)$ is bad, we find that 
\begin{equation}\label{eq:badest}
|2\alpha(A_2y_2,A_3y_3,A_4y_4) - \bs{u}'| \ll R|\bs{u}'|^{\frac{1}{2}}.
\end{equation}
As a result, the measure of the set of bad pairs is $O(R^3|\bs{u}'|^{-\frac{3}{2}})$. Taking the supremum over the bad pairs $(\bs{y},\alpha)$ and over $\bs{z}$ we get that there exists $|\omega| \asymp |\bs{u}'|$, and vectors $\bs{y}_0$, $\bs{z}_0$ such that
\begin{equation*}
\begin{split}
\mathcal{I}_j &\ll_{N} r^{-1-j}R^{-N} + R^3r^{\frac{1}{2}-j}\bigg|\int v_{\bs{y}_0}(x,\bs{z}_0)e(A_1\omega x^2+u(x))\, dx\bigg|,
\end{split}
\end{equation*}
since $\int |p(\alpha)| \, d\alpha \ll r^{-1}$. Setting $v_1(x) =  v_{\bs{y}_0}(x,\bs{z}_0)$ completes the proof of the lemma.
\end{proof}

To evaluate the ensuing integrals over $x_1$, we need the following results. The first lemma is a variant of the well-known second derivative estimate for exponential integrals. 
\begin{lemma}\label{tao}
Let $w(x)$ be a smooth function supported in $[1/2,2]$. Let $\phi(x)$ be a smooth function such that there exists $c > 0$ and $|\phi''(x)| \geq c$. Then
\begin{equation*}
\int w(x)e(\phi(x)) \, dx \ll c^{-\frac{1}{2}}.
\end{equation*}
\end{lemma}
\begin{proof}
Observe that
$$\int w(x)e(\phi(x)) \, dx = - \int_{1/2}^2 w'(x) \int_{1/2}^x e(\phi(y))\, dy\, dx.$$
By Proposition 2 in~\cite[\S VIII.1.2]{stein}, we have $\int_{1/2}^x e(\phi(y))\, dy \ll c^{-\frac{1}{2}}$. This completes the proof.
\end{proof} 
Using the above lemma, we prove the following two results.
\begin{lemma}\label{lemmasecond}
Let $w$ have compact support in $[1/2,2]$, and suppose that $A \neq 0$. Then for all $N \geq 0$ we have
\begin{equation}\label{eq:secondder}
\int w(x)e(Ax^2 + B\log x) \, dx \ll \min\left(1,|A|^{-\frac{1}{2}}\right).
\end{equation}
\end{lemma}
\begin{proof}
Suppose first that $B=0$. Then in this case the result follows immediately from Lemma~\ref{tao}. We will now consider the cases $AB > 0$ and $AB < 0$ separately. Let $\Psi(x) = Ax^2+ B\log x$. If $AB < 0$, then in this case $|\Psi''(x)| = |2A - \frac{B}{x^2}| \gg |A|$, and by appealing to Lemma~\ref{tao}, we are once again done. If $AB >0$, observe that $|\Psi'(x)| = |2Ax+B/x| \gg |A| + |B| \gg |A|$ in the interval $[1/2,2]$. Since $w$ is compactly supported, integrating by parts, we see that $$\int w(x) e(Ax^2+B\log x) \, dx \ll |A|^{-1}.$$ This completes the proof.
\end{proof} 

\begin{lemma}\label{csmallvoronoi}
Let $\alpha \geq 1.$ Define
$$I(\alpha, A) = \int_{1/2}^2 u(x)W_{k-1}(\alpha \sqrt{x})e(Ax^2+\tfrac{\alpha}{2\pi} \sqrt{x}) \, dx, $$ where $u(x)$ is a smooth function with support in $[1/2,2]$ and $A\alpha \neq 0$. Then we have 
$$I(\alpha, A) \ll (\alpha |A|)^{-\frac{1}{2}}.$$ 
\end{lemma}
\begin{proof}
We have
$$I(\alpha, A) = - \int_{1/2}^{2} (u(x)W_{k-1}(\alpha \sqrt{x}))' \int_{1/2}^x e(Ay^2+\tfrac{\alpha}{2\pi} \sqrt{y}) \, dy \, dx.$$ Arguing exactly as in the proof of Lemma~\ref{lemmasecond}, and using~\eqref{eq:boundforhankel} we get that
$$I(\alpha,A) \ll (|A|\alpha)^{-\frac{1}{2}}.$$
\end{proof}

We now have all the ingredients in place to prove the main results of this section.
\begin{proof}[Proof of Lemma~\ref{csmall}]
Recall from~\eqref{eq:june2} and~\eqref{eq:i1bound}, that it suffices to show that $I_2 \ll_{\ve} X^{\ve}.$ Suppose first that $|\bs{u}'| \ll r^{-\ve/2}$. Then we have $|\bs{u}'|^{1-\ve} \ll r^{-\ve}.$ In this case we make use of the trivial bound~\eqref{eq:i2bound} to get $$I_2 \ll r^{-1} \ll (r^{-1}|\bs{u}'|)^{\ve}r^{-1}|\bs{u}'|^{-1} \ll_{\ve} (r^{-1}|\bs{u}'|)^{\ve}|\bs{c}'|^{-1},$$

If, on the other hand, $|\bs{u}'| \gg r^{-\ve/2}$, we will apply Lemma~\ref{lemmauv} with $j=1$, $R = (r^{-1}|\bs{u}'|)^{\ve/12}$, $u(x) = \tfrac{t}{2\pi}\log x$ and $v(\bs{x}) = x_1^{\sigma+1}\widetilde{w}(\bs{x})$. As a result, we get that there exists a real number $\omega$ such that $|\omega| \asymp |\bs{u}'|$, a smooth function $v_1(x)$ with support in $[1/2,2]$ such that for any integer $N \geq 0$ we have 
\begin{equation*}
\begin{split}
I_2 &\ll_N r^{-2}R^{-N} + R^3 r^{-\frac{1}{2}}\int v_1(x)e(A_1\omega x^2 + \tfrac{t}{2\pi} \log x)\, dx \\
& \ll_N r^{-2}R^{-N} + R^3 r^{-\frac{1}{2}}|\omega|^{-\frac{1}{2}},
\end{split}
\end{equation*}
by Lemma~\ref{lemmasecond}. Taking $N$ large enough we see that $$I_2 \ll r^{-2}R^{-N} + R^3 r^{-\frac{1}{2}}|\bs{u}'|^{-\frac{1}{2}} \ll_{\ve} (r^{-1}|\bs{u}'|)^{\ve}|\bs{c}'|^{-\frac{1}{2}} \ll_{\ve} X^{\ve}.$$ This completes the proof of the lemma. 
\end{proof}
\begin{proof}[Proof of Lemma~\ref{stphasehb}]
Suppose first that $|\bs{u}'| \ll r^{-\ve/2}$. Then we have $|\bs{u}'|^{1-\ve} \ll r^{-\ve}.$ Using the bound for $I_{d,q}(\bs{c})$ in Lemma~\ref{iqcstrivialbound} we get
\begin{equation*}
\begin{split}
I_{d,q}(\bs{c}) &\ll \frac{1}{d}\left(1+\frac{\sqrt{c_1X}}{d}\right)^{-\frac{1}{2}} \\
&\ll \frac{(r^{-1}|\bs{u}'|)^{\ve}|\bs{u}'|^{-1}}{d} \left(1+\frac{\sqrt{c_1X}}{d}\right)^{-\frac{1}{2}} \\
&\ll_{\ve} \frac{X^{\ve}}{d}\left(1+\frac{\sqrt{c_1X}}{d}\right)^{-\frac{1}{2}}\left(\frac{q}{X}\right).
\end{split}
\end{equation*}

If, on the other hand, $|\bs{u}'| \gg r^{-\ve/2}$, we will treat the cases $c_1 \ll d^2/X$ and $c_1 \gg d^2/X$ separately. This dichotomy often arises in problems involving the Bessel function $J_{k-1}(x)$ and it is related to the behaviour of the function in the transition range $x \asymp_k 1$. To address this problem, we will apply Lemma~\ref{lemmauv} to two different pairs of functions $(u,v)$. 

If $c_1 \gg d^2/X$, we will use the expression in~\eqref{eq:11oct2} for $I_{d,q}(\bs{c})$ and study the two integrals on the RHS by applying Lemma~\ref{lemmauv} to each of them separately. To estimate the first integral (the analysis of the second integral is similar, so we will omit the details) we apply Lemma~\ref{lemmauv} with $j = 0, R = (r^{-1}|\bs{u}'|)^{\ve/12}$, $u(x) = 2\sqrt{c_1Xx}/d$ and $v(\bs{x}) = W_{k-1}(4\pi\sqrt{c_1Xx_1}/d)\widetilde{w}(\bs{x})$. For $j_k \geq 0$ it follows from~\eqref{eq:boundforhankel} that $$\frac{\partial^{j_2 +j_3 + j_4}v(\bs{x})}{\partial x_2^{j_2}\partial x_3^{j_3}\partial x_4^{j_4}} \ll_{k,w,j_2,j_3,j_4} \left(1+\frac{\sqrt{c_1Xx_1}}{d}\right)^{-\frac{1}{2}}.$$ Since all implicit constants in this article are allowed to depend on the weight function $w(\bs{x})$ and the weight $k$ of the cusp form $f(z)$, the pair of functions satisfy the hypotheses of Lemma~\ref{lemmauv}. As a result, we find that there exists a real number $\omega$ satisfying $|\omega| \asymp |\bs{u}'|$, a smooth function $v_1(x)$ with support in $[1/2,2]$, such that for any integer $N \geq 0$ we have \begin{equation*}
\begin{split}
I_{d,q}(\bs{c}) &\ll_N (dr)^{-1}R^{-N} + \\ 
&\quad \quad \frac{r^{\frac{1}{2}}R^3}{d} \int v_1(x)e(A_1\omega x^2+2\tfrac{\sqrt{c_1Xx}}{d}) \, dx. \\
\end{split}
\end{equation*}
Recall that in the proof of Lemma~\ref{lemmauv}, we have $v_1(x) = v_{\bs{y}_0}(x,\bs{z}_0)$ for some vectors $\bs{y}_0$ and $\bs{z}_0 \in \mathbf{R}^3$. By construction of the function $v_{\bs{y}_0}(x,\bs{z}_0)$, we see that $v_1(x) = v_2(x)W_{k-1}(4\pi\sqrt{c_1Xx}/d)$ for some smooth function $v_2(x)$ with support in $[1/2,2]$. Moreover, for all $j \geq 0$, it is easy to verify that $v_2^{(j)}(x) \ll_j 1$. Substituting this back into the above integral we get by Lemma~\ref{csmallvoronoi} that $$I_{d,q}(\bs{c}) \ll_N (dr)^{-1}R^{-N} + \frac{r}{d}\left(\frac{\sqrt{c_1X}}{d}\right)^{-\frac{1}{2}} R^3 \ll_{\ve} \frac{X^{\ve}}{d}\left(\frac{\sqrt{c_1X}}{d}\right)^{-\frac{1}{2}}\left(\frac{q}{X}\right),$$ by taking $N$ large enough. 

Finally, suppose that $c_1 \ll d^2/X$. In this case, we will use the expression for $I_{d,q}(\bs{c})$ in~\eqref{eq:11oct1} and apply Lemma~\ref{lemmauv} with $j =0 , R = (r^{-1}|\bs{u}'|)^{\ve/12}$, $u(x) = 0$ and $v(\bs{x}) = J_{k-1}(4\pi\sqrt{c_1Xx_1}/d)\widetilde{w}(\bs{x})$. It is once again clear that $u$ and $v$ satisfy the hypotheses of the lemma, and in this case we get that there exists a real number $\omega'$ such that $|\omega'| \asymp |\bs{u}'|$, a smooth function $v_3(x)$ with support in $[1/2,2]$ such that for any integer $N \geq 0$ we have
\begin{equation*}
\begin{split}
I_{d,q}(\bs{c}) &\ll_N (dr)^{-1}R^{-N} + \\ 
&\quad \quad \frac{r^{\frac{1}{2}}R^3}{d} \int v_3(x)e(A_1\omega x^2) \, dx. \\
\end{split}
\end{equation*}
Once again, we observe that $v_3(x) = v_4(x)J_{k-1}(4\pi\sqrt{c_1Xx}/d)$ for some smooth function $v_4(x)$ with support in $[1/2,2]$ such that $v_4^{(j)}(x) \ll_j 1$ for all $j \geq 0$. Arguing exactly as before, by using Lemma~\ref{tao} and the fact that the bound $xJ'_{k-1}(x) \ll 1$ holds for $x \ll 1$ (which follows from~\eqref{eq:boundforjbessel}), we find that $$I_{d,q}(\bs{c}) \ll_{\ve} \frac{q}{dX}X^{\ve}.$$ This completes the proof of the lemma.
\end{proof}
%\footnote{The functions $v_2(x)$ and $v_4(x)$ depend only on $F$ and $w$ and $|\bs{u}'|$}. The $\bs{u}'$ dependence is very minor. 
\section{Exponential sums}
We begin by establishing certain multiplicativity results for the exponential sums that we will encounter in the proof of Theorem ~\ref{mainthm}. Let $S_{d,q}(\bs{c})$ be the exponential sum in ~\eqref{eq:s2}. We have
\begin{lemma}\label{multiplicativitylemma}
Suppose that $d = u_1u_2$ and $q = v_1v_2$ with $(u_1v_1,u_2v_2) = 1$. Then the following holds, 
\begin{equation*}
\begin{split}
S_{d,q}(\bs{c}) &= S_{u_1,v_1}(\overline{u_2}^2c_1,\overline{v_2}\bs{c}')S_{u_2,v_2}(\overline{u_1}^2c_1,\overline{v_1}\bs{c}') \\
&= S_{u_1,v_1}(\overline{u_2}c_1,\bs{c}')S_{u_2,v_2}(\overline{u_1}c_1,\bs{c}').
\end{split}
\end{equation*}
\end{lemma}
\begin{proof}
Let $a=v_2a_1+v_1a_2$ where $a_i$ run modulo $v_i$. Let $\bs{b} = v_2\overline{v_2}\bs{s} + v_1\overline{v_1}\bs{t}$ where $s$ (respectively $t$) runs modulo $v_1$ (respectively $v_2$). Then
\begin{equation*}\label{eq:basicmult}
\begin{split}
e_{q}(aF(\bs{b})+\bs{b}'.\bs{c}') &= e_{v_1}(a_1F(\bs{s})+\bs{s}'.\overline{v_2}\bs{c}')e_{v_2}(a_2F(\bs{t})+\bs{t}'.\overline{v_1}\bs{c}').
\end{split}
\end{equation*}
Also, 
\begin{equation*}
\begin{split}
S(b_1,c_1;u_1u_2) &= S(\overline{u_2}s_1,\overline{u_2}c_1;u_1)S(\overline{u_1}t_1,\overline{u_1}c_1;u_2) \\
&= S(s_1,\overline{u_2}^2c_1;u_1)S(t_1,\overline{u_1}^2c_1;u_2).
\end{split}
\end{equation*}
This gives us the first multiplicativity statement. For the second, replace $\bs{s}$ (resp. $\bs{t}$) by $v_2\bs{s}$ (resp. $v_1\bs{t}$). 
\end{proof}

Let
\begin{equation}\label{eq:sq(n)}
A_q(\bs{c}) = \sideset{}{^*}\sum_{a \mod{q}}\sum_{\bs{b}' \mod{q}}e_q(aF(c_1,\bs{b}')+\bs{b}'.\bs{c}').
\end{equation}
\begin{lemma}\label{multforaq}
Let $q=v_1v_2$ with $(v_1,v_2)=1$. We have
\begin{equation*}
\begin{split}
A_q(\bs{c}) &= A_{v_1}(c_1,\overline{v_2}\bs{c}')A_{v_2}(c_1,\overline{v_1}\bs{c}')\\
&=A_{v_1}(\overline{v_2}c_1,\bs{c}')A_{v_2}(\overline{v_1}c_1,\bs{c}').
\end{split}
\end{equation*}
\end{lemma}
\begin{proof}
The proof is similar to the proof above: let $d=1$ and write $a$ and $\bs{b}'$ as in Lemma ~\ref{multiplicativitylemma}. We have,
\begin{equation*}
e_q(aF(c_1,\bs{b}')+\bs{b}'.\bs{c}') = e_{v_1}(a_1F(c_1,\bs{x}')+\bs{x}'.\overline{v_2}\bs{c}')e_{v_2}(a_2F(c_1,\bs{y}')+\bs{y}'.\overline{v_1}\bs{c}').
\end{equation*}
Therefore,
\begin{equation*}
A_q(\bs{c}) = A_{v_1}(c_1,\overline{v_2}\bs{c}')A_{v_2}(c_1,\overline{v_1}\bs{c}').
\end{equation*} 
The lemma follows by replacing $\bs{x}'$ by $v_2\bs{x}'$, $\bs{y}'$ by $v_1\bs{y}'$, and by replacing $a_1$ by $\overline{v_2}^2a_1$ and $a_2$ by $\overline{v_1}^2a_2.$ 
\end{proof}
Next, we recall some basic facts about Gauss sums. Let 
$$
\delta_n=\begin{cases} 0, & \mbox{ if }n\equiv 0\bmod{2}, \\ 1 & \mbox{ if } n\equiv 1\bmod{2}, \end{cases}
\quad
\epsilon_n=\begin{cases} 1, & \mbox{ if }n\equiv 1\bmod{4}, \\ i, & \mbox{ if } n\equiv 3\bmod{4}.
\end{cases}
$$
The following result is recorded in \cite[Lemma 3]{BB}, but it goes back to Gauss.
\begin{lemma} \label{gaussevlemma} 
Define
$$\mathcal{G}(s,t;q) = \sum_{b \mod{q}} e_q(sb^2+tb).$$ Suppose that $(s,q)=1$. Then 
$$ 
\mathcal{G}(s,t;q)=
\begin{cases}
\epsilon_q \sqrt{q}\left(\frac{s}{q}\right) e\left(-\frac{\overline{4s}t^2}{q}\right) & \text{ 
if $q$ is odd,}\\
2  \delta_t \epsilon_v\sqrt{v} \left(\frac{2s}{v}\right) e\left(-\frac{\overline{8s}t^{2}}{v}\right) &
\text{ if $q=2v$, with $v$ odd,}\\
(1+i) \epsilon_s^{-1} (1-\delta_t) \sqrt{q}\left(\frac{q}{s}\right) e\left(-\frac{\overline{s}t^{2}}{4q}\right) 
& \text{ 
if $4\mid q$.}
\end{cases}
$$
If $(s,q) \neq 1$, $\mathcal{G}(s,t;q) = 0$ unless $(s,q) \mid t$, in which case we have $$\mathcal{G}(s,t;q) = (s,q)\mathcal{G}\left(\frac{s}{(s,q)},\frac{t}{(s,q)};\frac{q}{(s,q)}\right).$$
\end{lemma}

\subsection{Analysis of $S_q(n)$}
Set
\begin{equation}\label{eq:definesqn}
S_q(n) = A_q(n,\bs{0}).
\end{equation}
We begin our analysis of $S_q(n)$ by showing that $S_q(n)$ depends only on $(n,q)$.
\begin{lemma}\label{referee}
We have
\begin{equation*}
S_q(n) = S_q((n,q)).
\end{equation*}
\end{lemma}
\begin{proof}
By Lemma~\ref{multforaq} we see that $S_q(n)$ is multiplicative in $q$. As a result, it is sufficient to verify the lemma in the case of prime-power moduli, $q=p^t$. 

If $v_p(n) \geq \lceil \frac{t}{2} \rceil$, we see that $$S_{p^t}(n) = S_{p^t}(0) = S_{p^t}((n,p^t)).$$ 

If $v_p(n) < \lceil \frac{t}{2} \rceil$, we have
\begin{equation*}
S_{p^t}(n) = \sumstar_{a \mod{p^t}}\sum_{\bs{b}' \mod{p^t}}e_{p^t}(aA_1(n,p^t)^2n_{p^t}^2+aF(0,\bs{b}')),
\end{equation*}
where $n_{p^t} = n/(n,p^t).$ Since $(p,n_{p^t})=1$, we can replace $a$ by $\overline{n_{p^t}}^2a$ and $\bs{b}'$ by $n_{p^t}\bs{b}'$. As a result, we find that
\begin{equation*}
S_{p^t}(n) = \sumstar_{a \mod{p^t}}\sum_{\bs{b}' \mod{p^t}}e_{p^t}(aA_1(n,p^t)^2+aF(0,\bs{b}')) = S_{p^t}((n,p^t)).
\end{equation*}
This completes the proof of the lemma.
\end{proof}
We will now turn to estimating $|S_q(n)|.$ 
\begin{lemma}\label{boundsforsqn}
We have the bound
$$S_q(n) \ll q^{\frac{5}{2}}.$$
If $q$ is square-free and $(q,2\Delta)=1$, we have
\begin{equation*}
|S_q(n)| \leq q^{2}.
\end{equation*}
\end{lemma}
\begin{proof}
Appealing to Lemma~\ref{multforaq}, it suffices to consider the case of prime-power moduli, $q=p^t$.
We have
\begin{equation*}
S_q(n) = \sumstar_{a \mod{p^t}}e_{p^t}(aA_1n^2)\prod_{i=2}^4\sum_{b_i \mod{p^t}}e_{p^t}(aA_ib_i^2).
\end{equation*}
By Lemma~\ref{gaussevlemma}, we have $$\sum_{b_i \mod{p^t}}e_{p^t}(aA_ib_i^2) \ll_{\Delta} p^{\frac{t}{2}}.$$ The first part of the lemma now follows by estimating the sum over $a$ trivially.

To establish the second claim, it suffices to consider the case where $q=p$ is a prime such that $(p,2\Delta)=1$. Proceeding as before, we have
\begin{equation*}
\begin{split}
S_p(n) &= \sumstar_{a \mod{p}}e_p(aA_1n^2)\prod_{i=2}^4\sum_{b_i \mod{p}}e_p(aA_ib_i^2)\\
&= p^{\frac{3}{2}}\epsilon_p^3\prod_{i=2}^4\jacobi{A_i}{p}\sumstar_{a \mod{p}}e_p(aA_1n^2)\jacobi{a}{p},
\end{split}
\end{equation*}
by Lemma~\ref{gaussevlemma}. The sum over $a$ is a quadratic Gauss sum, which we can evaluate explicitly. We have
\begin{equation*}
\sumstar_{a \mod{p}}e_p(aA_1n^2)\jacobi{a}{p} = \epsilon_p\jacobi{A_1n^2}{p}p^{\frac{1}{2}}.
\end{equation*}
The lemma now follows.
\end{proof}

\subsection{Exponential sums in the case where $F^{-1}(0,\bs{c}')=0$ and $\bs{c}' \neq \bs{0}$}\label{aq0}
Having estimated $S_q(n) = A_q(n,\bs{0})$, we will now relate it to the more general sum $A_q(n,\bs{c}')$ with $F^{-1}(0,\bs{c}')=0.$ In order to state the result, we will need the following definition.
\begin{definition}\label{a0}[Condition $A_0$]
Let $l_1,l_2,l_3$ and $l_4$ be non-zero integers, and let $q = \prod_{p^{k_p} \Vert q}$. We say that the tuple $(q;l_1,l_2,l_3,l_4)$ satisfies Condition $A_0$ if for each odd prime $p \mid q$, we have $k_p \geq \max\left\{v_p(l_1), v_p(l_2), v_p(l_3), v_p(l_4)\right\}$, if $p \nmid \Delta$, and if $k_p \geq \max\left\{2+v_p(l_1), v_p(l_2), v_p(l_3), v_p(l_4)\right\}$ for $p \mid \Delta$. If $p=2$, we require that $k_2 \geq 3+\max_{1 \leq i \leq 4}v_2(l_i).$ 
\end{definition}

\begin{lemma}\label{replacec'0}
Let $\bs{c}' \neq \bs{0} \in \mathbf{Z}^3$ and let $n \in \mathbf{N}$. Let $p$ be a prime and $q = p^k$. For $2 \leq i \leq 4$ suppose that $c_i \equiv 0 \pmod{p^{v_p(A_i)}}$, and that $F^{-1}(0,\bs{c}') = 0$. Let $A_q(n,\bs{c}')$ be as in ~\eqref{eq:sq(n)}. Suppose that $(q,A_2,A_3,A_4)$ satisfies condition $A_0$. Then
\begin{equation*}
A_q(n,\bs{c}') = A_q(n,\bs{0}) = S_q(n).
\end{equation*}
\end{lemma}
\begin{proof}
Let $a_i = v_p(A_i)$, as before. Then the sum over $\bs{b}'$ in ~\eqref{eq:sq(n)} is
\begin{equation*}\label{eq:ul}
\prod_{i=2}^4 p^{a_i} \sum_{b_i \mod{p^{k-a_i}}}e_{p^{k-a_i}}(aA_i/p^{a_i}b_i^2 + b_i c_i/p^{a_i}).
\end{equation*}
To ease notation, let $A_i' = A_i/p^{a_i}$ and let $c_i' = c_i/p^{a_i}$. If $p \neq 2$, by Lemma ~\ref{gaussevlemma} the above expression evaluates to
\begin{equation}\label{eq:ul2}
\prod_{i=2}^4 p^{\frac{k+a_i}{2}}\epsilon_{p^{k-a_i}}\left(\frac{aA_i'}{p^{k-a_i}}\right)e_{p^{k-a_i}}(-\overline{4a A_i'} c_i'^2),
\end{equation}
since by hypothesis $(A_i',p)=1$. Make a change of variables $\overline{a} \to -4A_2'A_3'A_4'b$, and observe that for $i=2,3,4$ we have $e_{p^{k-a_i}}(-\overline{4aA_i'}c_i'^2) = e_{p^{k-a_i}}(b
\prod_{\substack{2 \leq j \leq 4 \\ i \neq j}}A_j'c_i'^2)$. Consequently, the expression in ~\eqref{eq:ul2} is 
\begin{equation*}
\prod_{i=2}^4 p^{\frac{k+a_i}{2}}\epsilon_{p^{k-a_i}}\left(\frac{-4b\prod_{\substack{2 \leq j \leq 4 \\ i \neq j}}A_j'}{p^{k-a_i}}\right)e_{p^k}(p^{a_2}A_3'A_4'c_2'^2+p^{a_3}A_2'A_4'c_3'^2+p^{a_4}A_2'A_3'c_4'^2).
\end{equation*}
However, since $F^{-1}(0,\bs{c}') = A_3A_4c_2^2+\ldots+A_2A_3c_4^2 = 0$, we have $p^{a_2}A_3'A_4'c_2'^2 + p^{a_3}A_2'A_4'c_3'^2 + p^{a_4}A_2'A_3'c_4'^2 = 0$. Consequently, the exponential factor above is $=1,$ and we see that $A_q(n,\bs{c}')$ is independent of $\bs{c}'$ and this completes the proof for odd $p$. A similar argument works when $p=2.$
\end{proof}

\subsection{Auxillary estimates}
Recall the sum $S_{d,q}(\bs{c})$ from ~\eqref{eq:s2}. We begin by recording  a version of ~\cite[Lemma 27]{HB}. 
\begin{lemma}\label{lemma0}
Let $q=p^t$, $d=p^{\delta}$ with $t \geq 2$ and $\delta \leq t$. Suppose that $p \nmid 2\Delta$. Then $S_{d,q}(\bs{c})$ vanishes unless $p \mid F^{-1}(0,\bs{c}')$. 
\end{lemma}
\begin{proof}
In the expression 
\begin{equation*}
S_{d,q}(\bs{c}) = \sideset{}{^*}\sum_{z \mod{p^{\delta}}}e_{p^{\delta}}(c_1\overline{z})\sideset{}{^*}\sum_{a \mod{p^t}}\sum_{\bs{b} \mod{p^t}}e_{p^t}(aF(\bs{b})+\bs{b}'.\bs{c}'+p^{t-\delta}b_1z),
\end{equation*}
set $a = u + pv$ to see that
\begin{equation*}
\begin{split}
S_{d,q}(\bs{c}) &= \sideset{}{^*}\sum_{z \mod{p^{\delta}}}e_{p^{\delta}}(c_1\overline{z})\sideset{}{^*}\sum_{u \mod{p}}\sum_{\bs{b} \mod{q}}e_{p^t}(uF(\bs{b})+\bs{b}.(p^{t-\delta}z,\bs{c}'))\times \\ &\quad \quad \sum_{v \mod{p^{t-1}}}e_{p^{t-1}}(vF(\bs{b})) \\
&= p^{t-1}\sideset{}{^*}\sum_{z \mod{p^{\delta}}}e_{p^{\delta}}(c_1\overline{z})\sideset{}{^*}\sum_{u \mod{p}}\sum_{\substack{\bs{b} \mod{p^t} \\ F(\bs{b}) \equiv 0 \mod{p^{t-1}}}}e_{p^t}(uF(\bs{b})+\bs{b}.(p^{t-\delta}z,\bs{c}')).
\end{split}
\end{equation*}
Writing $\bs{b} = \bs{x} + p^{t-1}\bs{y}$, we get that
\begin{equation*}
\begin{split}
S_{d,q}(\bs{c}) &= p^{t-1}\sideset{}{^*}\sum_{z \mod{p^{\delta}}}e_{p^{\delta}}(c_1\overline{z})\times \\ &\quad\quad \sideset{}{^*}\sum_{u \mod{p}}\sum_{\substack{\bs{x} \mod{p^{t-1}} \\ F(\bs{x}) \equiv 0 \mod{p^{t-1}}}}e_{p^t}(uF(\bs{x}) + p^{t-\delta}x_1z +\bs{x}'.\bs{c}')\times \\ &\quad \quad \quad \sum_{\bs{y} \mod{p}}e_p(\bs{y}.(u\nabla F(\bs{x}) + (p^{2t-\delta-1}z,\bs{c}')).
\end{split}
\end{equation*} 
As $2t \geq 2+\delta$, the sum over $\bs{y}$ vanishes unless $\nabla F(\bs{x}) \equiv -\overline{u}(0,\bs{c}') \pmod{p}$. Since $p \nmid 2\Delta$, this is the same as the condition $\bs{x} \equiv -\overline{2u}M^{-1}(0,\bs{c}') \pmod{p}$. Observe that this forces $F(\bs{x}) \equiv \overline{4u^2}F^{-1}(0,\bs{c}') \pmod{p}$. Consequently, the sum over $\bs{x}$ vanishes unless $F^{-1}(0,\bs{c}') \equiv 0 \pmod{p}$, and the lemma follows.
\end{proof}

Let 
\begin{equation*}\label{eq:tqr}
T_q(r) = \sideset{}{^*}\sum_{a \mod{q}}\sum_{\bs{b} \mod{q}}e_q(aF(\bs{b}) + b_1r + \bs{b}'.\bs{c}').
\end{equation*}
and 
\begin{equation}\label{eq:tq}
T_q = \sum_{r \mod{q}}|T_q(r)|.
\end{equation}
In the proof of Theorem ~\ref{mainthm}, we will need good control on the average order of $T_q$. We begin with the following observation.
\begin{lemma}\label{lemmarefereevanishing}
Let $t \geq 2$ be an integer. Suppose that $p \nmid 2\Delta$ and that $p \mid r$, then we have that $T_{p^t}(r) = 0$ unless $p \mid F^{-1}(0,\bs{c}')$.
\end{lemma}
\begin{proof}
To see this, we argue as in the proof of Lemma ~\ref{lemma0} to see that
\begin{equation*}
\begin{split}
T_{p^t}(r) &= p^{t+3}\sumstar_{u \mod{p}}\sum_{\bs{x} \mod{p^{t-1}}}e_{p^t}(uF(\bs{x}) + rx_1 + \bs{x}'.\bs{c}'),
\end{split}
\end{equation*}
where the $\bs{x}$-sum is also subject to the conditions $F(\bs{x}) \equiv 0 \pmod{p^{t-1}}$ and $2M\bs{x} \equiv -\overline{u}(r,\bs{c}') \pmod{p}$, and $M$ is the matrix associated to the quadratic form $F$. It is then easy to see that if $p \mid r$, then $p \mid x_1$, and this in turn implies that $p \mid F^{-1}(0,\bs{c}'),$ as claimed.
\end{proof}
\begin{lemma}\label{c'0}
Suppose that $F^{-1}(0,\bs{c}') \neq 0$ and $|\bs{c}'| \ll X^{\ve}$. Then
\begin{equation*}
\sum_{q \leq X}T_q \ll_{\ve} X^{4+\ve}.
\end{equation*}
\end{lemma}
\begin{proof}
Observe that $T_q$ is multiplicative in $q$. Write $q = uv$ where $u$ is square-free and $v$ is square-full. Let $N = 2|\Delta||F^{-1}(0,\bs{c}')|.$ Further factorise $v = v_1v_2$, with the property that $(v_1,N)=1$ and $p \mid N$ for any prime $p$ that divides $v_2$. Thus we are led to estimating $T_u$, $T_{v_1}$ and $T_{v_2}$ individually.

If $p \nmid 2\Delta$, it follows from ~\cite[Lemma 26]{HB} that $$T_p = p^2 \sum_{r_p \mod{p}}|c_p(F^{-1}(r_p,\bs{c}'))| \leq 3p^3.$$ Furthermore, if $p \mid 2\Delta$, observe that $T_p \ll_F 1.$ Hence we have 
\begin{equation*}
T_u \ll u^3 3^{\omega(u)}.
\end{equation*}
By ~\cite[Lemma 25]{HB} we see that 
\begin{equation*}
T_{v_2} \ll v_2^4.
\end{equation*}
By Lemma~\ref{lemmarefereevanishing} it follows that
\begin{equation*}T_{v_1} = \sumstar_{r \mod{v_1}}|T_{v_1}(r)|.\end{equation*}
By applying Lemma ~\ref{gaussevlemma} to each term in $T_{v_1}$ we get
\begin{equation*}
\begin{split}
T_{v_1} &= v_1^2 \sumstar_{r \mod{v_1}}|c_{v_1}(\overline{A_1}r^2+F^{-1}(0,\bs{c}'))| \\
&\leq v_1^2 \sumstar_{r \mod{v_1}}(v_1,r^2+F^{-1}(0,\bs{c}')) \ll_{\ve} X^{\ve}v_1^3,
\end{split}
\end{equation*}
since $|\bs{c}'| \ll X^{\ve}$. 
As a result,
\begin{equation*}
\begin{split}
\sum_{q \leq X}T_q &\ll X^{\ve}\sum_{\substack{v_2 \leq X \\ p \mid v_2 \implies p \mid N}}v_2^4\sum_{uv_1 \leq X/v_2}(uv_1)^3 \ll_{\ve} X^{4+\ve}
\end{split}
\end{equation*}
since $$\sum_{\substack{v \leq X \\ p \mid v \implies p \mid N}} 1 \ll_{\ve} (NX)^{\ve}.$$ This completes the proof of the lemma. 
\end{proof}
\begin{remark}
Notice that we do not need the condition $F^{-1}(0,\bs{c}') \neq 0$ to estimate the sum over the square-free part. However, we have used this fact to restrict the number of terms in the $v$-sum. Without this observation, Lemma ~\ref{c'0} would only hold with the weaker upper bound $O(X^{\frac{9}{2}+\ve}).$
\end{remark}

Next, we analyse the sum $S_{d,q}(\bs{c})$. Observe that Lemma ~\ref{multiplicativitylemma} shows that it suffices to consider the case where $q=p^k$ is a prime power. Lemma ~\ref{lemma0} shows that for $(p,2\Delta)=1$, if $p^2 \mid q$ then $S_{d,q}(\bs{c})$ vanishes unless $p \mid F^{-1}(0,\bs{c}').$ If $d =1$, then $S_{1,q}(\bs{c}) = S_{1,q}(0,\bs{c}')$. 

Let $d = p^{\delta}$ and $q = p^{\kappa}$, with $\delta \leq \kappa$. Recall that
\begin{equation*}
\begin{split}
S_{d,q}(\bs{c}) &= \sideset{}{^*}\sum_{a \mod{p^{\kappa}}}\sum_{\bs{b} \mod{p^{\kappa}}}e_{p^{\kappa}}(aF(\bs{b})+\bs{b}'.\bs{c}')S(b_1,c_1;p^{\delta}).
\end{split}
\end{equation*}
Then, 
\begin{equation*}
\begin{split}
S_{d,q}(\bs{c}) &= \sumstar_{x \mod{p^{\delta}}}e_{p^{\delta}}(c_1\overline{x})\sumstar_{a \mod{p^{\kappa}}}\sum_{\bs{b} \mod{p^{\kappa}}}e_{p^\kappa}(aF(\bs{b})+p^{\kappa-\delta}b_1x+\bs{b}'.\bs{c}').
\end{split}
\end{equation*} 
Using Lemma~\ref{gaussevlemma} to evaluate the Gauss sums, and estimating the sums over $a$ and $x$ trivially, we obtain the bound
\begin{equation}\label{eq:sdqtrivialestimate}
S_{d,q}(\bs{c}) \ll_{\Delta} p^{4\kappa}.
\end{equation}

Suppose next that $p \neq 2$. For $1 \leq i \leq 4$, let $p^{a_i} = (A_i,p^{\kappa}).$ By Lemma ~\ref{gaussevlemma} we see that $S_{d,q}(\bs{c})$ vanishes unless $c_i \equiv 0 \pmod{p^{a_i}}$ and in this case,
\begin{equation}\label{eq:sdqb'}
\begin{split}
S_{d,q}(\bs{c}) &= p^{\frac{3k+a_2+a_3+a_4}{2}}\prod_{i=2}^4\epsilon_{p^{\kappa-a_i}}\left(\frac{A_i/p^{a_i}}{p^{\kappa-a_i}}\right)\times \\ &\quad \sideset{}{^*}\sum_{a \mod{p^{\kappa}}}\prod_{i=2}^4\left(\frac{a}{p^{\kappa - a_i}}\right)e_{p^{\kappa - a_i}}(-\overline{4aA_i/p^{a_i}}(c_i/p^{a_i})^2) \times \\ 
&\quad \quad \sum_{b_1 \mod{p^{\kappa}}}e_{p^{\kappa}}(aA_1b_1^2)S(b_1,c_1;p^{\delta}).
\end{split}
\end{equation}

If $(p,\Delta)=1$, and $d = q = p$, we have $a_i=0$, and notice that the sum over $b_1$ in ~\eqref{eq:sdqb'} is
\begin{equation*}
\begin{split}
&= \sideset{}{^*}\sum_{x \mod{p}}e_p(c_1\overline{x})\sum_{b_1 \mod{p}}e_p(aA_1b_1^2+b_1x) \\
&= \epsilon_p p^{\frac{1}{2}}\left(\frac{aA_1}{p}\right)\sideset{}{^*}\sum_{x \mod{p}}e_p(c_1 \overline{x} - \overline{4aA_1}x^2).
\end{split}
\end{equation*}
Consequently,
\begin{equation*}
\begin{split}
S_{p,p}(\bs{c}) &= p^2\left(\frac{\Delta}{p}\right) \sideset{}{^*}\sum_{a,x \mod{p}}e_p(c_1\overline{x} - \overline{4a}F^{-1}(x,\bs{c}')) \\
&= p^2\left(\frac{\Delta}{p}\right)\left\{\phi(p)\sideset{}{^*}\sum_{\substack{x \mod{p}\\ F^{-1}(x,\bs{c}') \equiv 0 \mod{p}}}e_p(c_1\overline{x}) - \sideset{}{^*}\sum_{\substack{x \mod{p} \\ F^{-1}(x,\bs{c}') \not\equiv 0 \mod{p}}}e_p(c_1\overline{x})\right\}.
\end{split}
\end{equation*}
Hence $|S_{p,p}(\bs{c})| \leq 3p^3.$ We summarise our findings in the following result.
\begin{lemma}\label{lemmasdqorder}
Suppose that $d \mid q = p^{\kappa}$. We then have $S_{d,q}(\bs{c}) \ll_{\Delta} q^4.$ If $(p,2\Delta)=1$ and $q=p$, then $S_{1,p}(\bs{c}) \ll p^2(p,F^{-1}(0,\bs{c}'))$, and 
$|S_{p,p}(\bs{c})| \leq 3p^{3}.$ If $p \mid 2\Delta$, then $S_{1,p}(\bs{c}) \ll_{\Delta} 1$ and $S_{p,p}(\bs{c}) \ll_{\Delta} 1.$ \end{lemma}
%\begin{remark}
%Under the hypothesis of the above lemma, one can show that $S_{d,q}(\bs{c}) \ll q^{\frac{7}{2}+\frac{1}{4}}$ if $p=2$, and that $S_{d,q}(\bs{c}) \ll q^{\frac{7}{2}}$ if $p$ is odd.  
%\end{remark}

\section{Proof of Theorem ~\ref{mainthm}}
It follows from ~\eqref{eq:prevoronoi} and Lemma ~\ref{lemmatruncatec'} that for any $\ve > 0$, 
\begin{equation*}
\begin{split}
N(\lambda;X) = c_QX\sum_{q \ll X}q^{-3}&\sum_{|\bs{c}'| \ll X^{\ve}}\sideset{}{^*}\sum_{a \mod{q}}\sum_{\bs{b} \mod{q}}e_q(aF(\bs{b})+\bs{b}'.\bs{c}') \\
&\times \sum_{c_1 \equiv b_1 \mod{q}}\lambda(c_1)I_q(\bs{c}) + O(1).
\end{split}
\end{equation*}
Our task now is to show that the right hand side is $o(X^2)$. The analysis of the exponential sum is predicated on the vanishing or non-vanishing of $F^{-1}(0,\bs{c}')$. Define the sets 
\begin{equation*}
\begin{split}
\mathcal{C}_0 &= \left\{\bs{c}' \in \mathbf{Z}^3, |\bs{c}'| \ll X^{\ve} : F^{-1}(0,\bs{c}') = 0 \right\}, \\
\mathcal{C}_1 &= \left\{\bs{c}' \in \mathbf{Z}^3, |\bs{c}'| \ll X^{\ve} : F^{-1}(0,\bs{c}') \neq 0 \right\}.
\end{split}
\end{equation*}
For $i = 0,1$ let $N^{(i)}(\lambda;X)$ denote the contribution from $\bs{c}' \in \mathcal{C}_i.$ We will show that there exists a $\delta > 0$ such that $N^{(i)}(\lambda; X) \ll X^{2-\delta}$. We start with $N^{(0)}(\lambda;X)$. 
\subsection{Contribution from $N^{(0)}(\lambda;X)$}
Let $|\bs{c}'| \ll X^{\ve}$ such that $F^{-1}(0,\bs{c}')=0$. Recall the sum $A_q(\bs{c})$ from ~\eqref{eq:sq(n)}. Set 
\begin{equation*}\label{eq:n0flc'}
N^{(0)}(\lambda,\bs{c}';X) = \sum_{q \ll X}q^{-3}\sum_{c_1=1}^{\infty}\lambda(c_1)A_q(\bs{c})I_q(\bs{c}).
\end{equation*}
Then $N^{(0)}(\lambda;X) = c_QX\textstyle\sum_{\bs{c}' \in \mathcal{C}_0}N^{(0)}(\lambda,\bs{c}';X).$ 

We begin by writing $q = rs$, a product of coprime integers, as follows. Recalling Condition $A_0$ (Definition~\ref{a0}), let
\begin{equation*}
r= \prod_{\substack{p^k \Vert q \\ (p^k;A_1,\ldots,A_4) \text{ satisfies } \\  \text{ Condition } A_0}} p^k
\end{equation*}
be the greatest divisor of $q$ that satisfies Condition $A_0$. By Lemma ~\ref{multforaq} we have $A_q(\bs{c}) = A_r(\overline{s}c_1,\bs{c}')A_s(\overline{r}c_1,\bs{c}')$. Lemma ~\ref{gaussevlemma} shows that $A_r(\bs{c})$ vanishes unless $c_i \equiv 0 \pmod{p^{v_p(A_i)}}$, for $2 \leq i \leq 4$, so without loss of generality, we may assume that $\bs{c}'$ satisfies this condition. By construction, Lemma ~\ref{replacec'0} applies to the sum $A_r(\overline{s}c_1,\bs{c}')$. Notice also that $s \ll |\Delta| \ll_F 1$. As a result, we have
\begin{equation}\label{eq:splitintocongruence}
\begin{split}
\sum_{c_1=1}^{\infty}\lambda(c_1)A_q(\bs{c})I_q(\bs{c}) &= \sum_{\sigma \mod{s}}A_s(\overline{r}\sigma,\bs{c}')\sum_{c_1 \equiv \sigma \mod{s}}\lambda(c_1)A_r(\overline{s}c_1,\bs{0})I_q(\bs{c}) \\
&= \sum_{\sigma \mod{s}}A_s(\overline{r}\sigma,\bs{c}')\sum_{c_1 \equiv \sigma \mod{s}}\lambda(c_1)A_r(c_1,\bs{0})I_q(\bs{c}) \\
&= \sum_{\sigma \mod{s}}A_s(\overline{r}\sigma,\bs{c}')\Sigma_r(\sigma,s),
\end{split}
\end{equation}
say. By Lemma~\ref{referee} we have
\begin{equation*}
\begin{split}
\Sigma_r(\sigma,s) &= \sum_{\rho \mid r}\sum_{\substack{c_1 \equiv \sigma \mod{s} \\ (c_1,r)=\rho}}\lambda(c_1)S_r(c_1)I_q(\bs{c}) \\
&= \sum_{\rho \mid r}S_r(\rho)\sum_{\substack{\rho c_1 \equiv \sigma \mod{s} \\ (c_1,r/\rho)=1}}\lambda(\rho c_1)I_q((\rho c_1,\bs{c}')).
\end{split}
\end{equation*}
Observe that $(\rho,s)=1$. Clearing denominators, and using multiplicative characters to cut out the congruence condition $c_1 \equiv \overline{\rho}\sigma \pmod{s}$, we see that
\begin{equation}\label{eq:complicatedsum1}
\Sigma_r(\sigma,s) = \frac{1}{\phi(\hat{\sigma})}\sum_{\rho \mid r}S_r(\rho)\sum_{\chi \mod{\hat{s}}}\overline{\chi}(\overline{\rho}\hat{\sigma})\sum_{(c_1,r/\rho)=1}\chi(c_1)\lambda((\sigma,s)\rho c_1)I_q((\sigma,s)\rho c_1,\bs{c}'),
\end{equation}
where $\hat{s} = s/(\sigma,s)$ and $\hat{\sigma} = \sigma/(\sigma,s)$. To analyse the inner sum, we need the following
\begin{proposition}\label{newprop}
Let $\bs{c}' \neq \bs{0}$. Let $\chi$ be a Dirichlet character modulo $D$, and let $\theta,\kappa$ be positive integers. Then there exists $A > 0$ such that for all $\ve >0$ we have
\begin{equation*}
\sum_{(n,\kappa)=1} \chi(n)\lambda(\theta n)I_q(\theta n, \bs{c}') \ll_{\ve} (\theta \kappa)^{\ve}D^A\frac{X^{5/6+\ve}}{\theta^{\frac{1}{2}}q^{\frac{1}{3}}}.
\end{equation*}
\end{proposition}
\begin{proof}
Let $S(\chi,\theta)$ be the sum in question. Define the Dirichlet series 
\begin{equation*}
F_{\chi,\theta}(s) = \sum_{(n,\kappa)=1}^{\infty}\frac{\chi(n)\lambda(\theta n)}{n^s}.
\end{equation*}
Since $f$ is a newform, $$\lambda(mn) = \sum_{d \mid (m,n)}\mu(d)\lambda(m/d)\lambda(n/d).$$ As a result, for $\sigma >1$
\begin{equation}\label{eq:factorisation}
\begin{split}
F_{\chi,\theta}(s) &= \sum_{\substack{\beta \mid \theta \\ (\beta,\kappa)=1}}\frac{\mu(\beta)\chi(\beta)\lambda(\theta/\beta)}{\beta^s}\sum_{(n,\kappa)=1} \frac{\chi(n)\lambda(n)}{n^s} \\
&= P(\chi,\theta,\kappa)L(s,f\otimes \chi),
\end{split}
\end{equation}
where $$P(\chi, \theta, \kappa) = \prod_{\substack{p^l \Vert \theta \\ (p,\kappa)=1}} \left(\lambda(p^l) - \frac{\chi(p)\lambda(p^{l-1})}{p^s}\right)\prod_{p \mid \kappa}\left(1-\frac{\lambda(p)\chi(p)}{p^s} + \frac{\chi^2(p)}{p^{2s}}\right).$$
Recall from ~\eqref{eq:eulerproduct} that $L(s,f\otimes \chi)$ has an Euler product,
and if $\chi^*$ is the primitive character, of conductor $D^*$, say, that induces $\chi$, observe that
\begin{equation*}
L(s, f\otimes \chi) = \prod_{p \mid D}\left(1-\frac{\lambda(p)\chi^{*}(p)}{p^s} + \frac{\chi^*(p)^2}{p^{2s}}\right)L(s, f\otimes \chi^*).
\end{equation*} 
Applying ~\eqref{eq:phraglind} to $L(s,f\otimes \chi^{*})$ for $\frac{1}{2} \leq \sigma \leq 1$ we get that \begin{equation}\label{eq:phlind}F_{\chi,\theta}(s) \ll_{\ve} (\theta\kappa)^{\ve}(D{^*}(1+|t|))^{1-\sigma+\ve},\end{equation} and by ~\eqref{eq:weyl} we get that
\begin{equation}\label{eq:weylforf}
F_{\chi,\theta}(s) \ll_{\ve} (\theta\kappa)^{\ve}D{^*}^A(1+|t|)^{\frac{1}{3}+\ve},
\end{equation}
when $\sigma = \frac{1}{2}.$ Recall the integral $I_q(\bs{c}',s)$ from ~\eqref{eq:iqcs}. By the Mellin inversion theorem, we have for any $\sigma > 1$ that
\begin{equation*}
S(\chi,\theta) = \frac{1}{2\pi i}\int_{(\sigma)} \left(\frac{X}{\theta}\right)^s F_{\chi,\theta}(s)I_q(\bs{c}',s) \, ds. 
\end{equation*}

Our intention is to move the line of integration to $\Re(s) = \frac{1}{2}$ and to deploy the estimate~\eqref{eq:weylforf}. To execute this, we will first truncate the integral. Fix $\ve > 0$ and set $T=r^{-1}X^{\ve}$. %Write
%\begin{equation*}
%\begin{split}
%\int_{(1+\alpha)} \left(\frac{X}{b}\right)^s F_{\chi,b}(s)I_q(\bs{c}',s) \, ds &= \int_{1+\alpha-iT}^{1+\alpha+iT} \left(\frac{X}{\widetilde{q_1}\sqrt{q_2}}\right)^s F_{\chi,b}(s)I_q(\bs{c}',s) \, ds \, + \\
%&\quad \int_{1+\alpha+iT}^{1+\alpha+\i\infty} + \int_{1+\alpha-i\infty}^{1+\alpha-iT}.
%\end{split}
%\end{equation*}
By ~\eqref{eq:iqhparts} and the fact that $F_{\chi,\theta}(s) \ll 1$ in the region $\sigma >1$, we have
\begin{equation*}
\begin{split}
\int_{(\sigma)} \left(\frac{X}{\theta}\right)^s F_{\chi,\theta}(s)I_q(\bs{c}',s) \, ds &= \int_{\sigma-iT}^{\sigma+iT} \left(\frac{X}{\theta}\right)^s F_{\chi,\theta}(s)I_q(\bs{c}',s) \, ds \, + \\
&\quad O\left(X^{\sigma}r^{-N-1}\int_{|t| \geq T}|t|^{-N} \, dt\right).
\end{split}
\end{equation*}
The error term is $$O\left(X^{\sigma}r^{-2}X^{(1-N)\ve}\right).$$ Choosing $N$ large enough we get that
\begin{equation*}
\begin{split}
\int_{(\sigma)} \left(\frac{X}{\theta}\right)^s F_{\chi,\theta}(s)I_q(\bs{c}',s) \, ds &= \int_{\sigma-iT}^{\sigma+iT} \left(\frac{X}{\theta}\right)^s F_{\chi,\theta}(s)I_q(\bs{c}',s) \, ds \, + O_A(X^{-A}).
\end{split}
\end{equation*}
By ~\eqref{eq:phlind} and ~\eqref{eq:iqhparts} the horizontal integrals are bounded as follows,
\begin{equation*}
\begin{split}
\int_{\sigma\pm iT}^{\frac{1}{2} \pm iT} \left(\frac{X}{\theta}\right)^s F_{\chi,\theta}(s)I_q(s) \, ds
&\ll_{\ve} (\theta\kappa)^{\ve}D^{*}X^{\sigma-(N-2)\ve}r^{-\frac{1}{2}}.
\end{split}
\end{equation*}
Once again, choosing $N$ large enough, we get that
\begin{equation*}
\int_{(\sigma)} \left(\frac{X}{\theta}\right)^s F_{\chi,\theta}(s)I_q(\bs{c}',s) \, ds = \int_{\frac{1}{2}-iT}^{\frac{1}{2}+iT} \left(\frac{X}{\theta}\right)^s F_{\chi,\theta}(s)I_q(\bs{c}',s) \, ds \, + O_A(X^{-A}).
\end{equation*}
By Lemma~\ref{csmall}, and~\eqref{eq:weylforf} we have 
\begin{equation*}
\begin{split}
\int_{\frac{1}{2}-iT}^{\frac{1}{2}+iT} \left(\frac{X}{\theta}\right)^s F_{\chi,\theta}(s)I_q(\bs{c}',s) \, ds 
&\ll_{\ve} (\theta\kappa)^{\ve}D^A\left(\frac{X}{\theta}\right)^{\frac{1}{2}}T^{\frac{1}{3}+\ve}\int |I_q(\bs{c}',s)| \, ds \\
&\ll_{\ve} (\theta\kappa)^{\ve}D^A\left(\frac{X}{\theta}\right)^{\frac{1}{2}}T^{\frac{1}{3}+\ve} \ll_{\ve} (\theta\kappa)^{\ve}D^A\frac{X^{\frac{5}{6}+\ve}}{\theta^{\frac{1}{2}}q^{\frac{1}{3}}}.
\end{split}
\end{equation*}
This completes the proof of the proposition.
\end{proof}
Applying Proposition ~\ref{newprop} to the inner sum in ~\eqref{eq:complicatedsum1} we get that
\begin{equation}\label{eq:referee1}
\Sigma_r(\sigma,s) \ll_{\ve} X^{\frac{5}{6}+\ve}\sum_{\rho \mid r}|S_r(\rho)|,
\end{equation}
since $\cond(\chi) \ll s \ll_F 1$. Set
\begin{equation*}
\mathfrak{s}(r)= \begin{cases} r^2 &\text{if $r$ is square-free and $(r,2\Delta)=1$,}\\
r^{\frac{5}{2}} &\text{otherwise.} \end{cases}
\end{equation*}
By~\eqref{eq:splitintocongruence},~\eqref{eq:referee1} and Lemma~\ref{boundsforsqn} we have
\begin{equation*}
\begin{split}
N^{(0)}(\lambda,\bs{c}';X) &\ll_{\ve} X^{\frac{5}{6}+\ve}\sum_{q \ll X} \frac{\mathfrak{s}(r)}{q^{3}},
\end{split}
\end{equation*}
where $r \mid q$ is the largest divisor of $q$ that such that $(r;A_1,\ldots,A_4)$ satisfies Condition $A_0$. Let $s=q/r$, as before, and write $r = uv$ where $u$ is square-free and $v$ is square-full. We further factorise $u=u_0u_1$, where $(u_0,2\Delta)=1$ and $u_1 \mid (2\Delta)^{\infty}.$ Since $s \ll 1$,
\begin{equation*}
\begin{split}
N^{(0)}(\lambda,\bs{c}';X) &\ll_{\ve} X^{\frac{5}{6}+\ve}\sum_{v \ll X} \frac{1}{v^{\frac{1}{2}}} \sum_{u_0u_1 \ll X/v}\frac{1}{u_0u_1^{\frac{1}{2}}} \\
&\ll_{\ve} X^{\frac{5}{6}+\ve},
\end{split}
\end{equation*}
since the number of square-full integers $v \leq X$ is $O(X^{\frac{1}{2}}),$ and the number of integers $u_1 \leq X$ such that $u_1 \mid (2\Delta)^{\infty}$ is $O_{\ve}((|\Delta|X)^{\ve})$. Summing over $\bs{c}' \in \mathcal{C}_0$ we obtain the bound 
\begin{equation}\label{eq:n0lambdax}
N^{(0)}(\lambda;X) \ll_{\ve} X^{2-\frac{1}{6}+\ve}.
\end{equation}
\subsection{Contribution from $N^{(1)}(\lambda;X)$}

Next we examine $N^{(1)}(\lambda;X)$. Let $$N^{(1)}(\lambda,\bs{c}';X) = \sum_{q\ll X}q^{-3}N_q^{(1)}(F,\lambda,\bs{c}'),$$ so that 
\begin{equation*}
\begin{split}
N^{(1)}(\lambda;X) = c_QX\sum_{\bs{c}' \in \mathcal{C}_1}N^{(1)}(\lambda,\bs{c}';X).
\end{split}
\end{equation*}
For $\bs{c}' \in \mathcal{C}_1$ define
\begin{equation}\label{eq:refereenq1}
\begin{split}
N_q^{(1)}(\lambda,\bs{c}';X) &= \sideset{}{^*}\sum_{a \mod{q}}\sum_{\bs{b} \mod{q}}e_q(aF(\bs{b})+\bs{b}'.\bs{c}')\\ &\quad \quad \times \sum_{c_1 \equiv b_1 \mod{q}}\lambda(c_1)I_q(\bs{c}).
\end{split}
\end{equation}
By ~\eqref{eq:postvoronoi} and Lemma ~\ref{truncatem1} we see that
\begin{equation}\label{eq:nq1}
N_q^{(1)}(\lambda, \bs{c}';X)= \frac{X}{q}\sum_{d \mid q}\sum_{c_1 \ll X^{1+\ve}/(q/d)^2}\lambda(c_1) S_{d,q}(\bs{c})I_{d,q}(\bs{c}) + O_N(X^{-N}).
\end{equation}
Write $q=uv$ where $(u,2\Delta)=1$ is square-free, and $v$ is square-full and is composed of primes dividing $N=2|\Delta F^{-1}(0,\bs{c}')|$. By Lemma ~\ref{lemmasdqorder} we have
$$S_{d,q}(\bs{c}) \ll_{\ve} (u,F^{-1}(0,\bs{c}'))u^{3+\ve}v^4.$$
Applying Lemma ~\ref{stphasehb} to estimate $I_{d,q}(\bs{c}),$ obtain the bound
\begin{equation*}
\begin{split}
\sum_{c_1 \ll X^{1+\ve}/(q/d)^2} |\lambda(c_1)I_{d,q}(\bs{c})| &\ll X^{\ve}\left(\frac{q}{d^{\frac{1}{2}}X^{\frac{5}{4}}}\right)\sum_{c_1 \ll X^{1+\ve}/(q/d)^2}c_1^{-\frac{1}{4}} \\
&\ll X^{\ve}\left(\frac{q}{d^{\frac{1}{2}}X^{\frac{5}{4}}}\right)\left(\frac{X^{\frac{3}{4}}}{(q/d)^{\frac{3}{2}}}\right) \\
&\ll \frac{q^{\frac{1}{2}}X^{\ve}}{X^{\frac{1}{2}}}.
\end{split}
\end{equation*}
Inserting our bound for $S_{d,q}(\bs{c})$ into ~\eqref{eq:nq1} we have shown
\begin{proposition}\label{pr1}
Suppose that $1 \leq |\bs{c}'| \ll X^{\ve}$.  With notation as above, we have
 \begin{equation*}
N_q^{(1)}(\lambda, \bs{c}';X)\ll_{\ve} (u,F^{-1}(0,\bs{c}'))u^3v^4 \frac{X^{\frac{1}{2}+\ve}}{q^{\frac{1}{2}}}.
\end{equation*}
\end{proposition}

We can also estimate the sum over $c_1$ in ~\eqref{eq:refereenq1} using partial summation: employing additive characters to detect the congruence condition $c_1 \equiv b_1 \pmod{q}$ we find that
\begin{equation*}
\begin{split}
N_q^{(1)}(\lambda, \bs{c}';X)&= \frac{1}{q}\sideset{}{^*}\sum_{a \mod{q}}\sum_{\substack{\bs{b} \mod{q} \\ r \mod{q}}}e_q(aF(\bs{b})+b_1r+\bs{b}'.\bs{c}')\sum_{c_1=1}^{\infty}\lambda(c_1)e_q(-rc_1)I_q(\bs{c}) \\ 
&\leq \frac{1}{q}\sum_{r \mod{q}}\lvert \sideset{}{^*}\sum_{a \mod{q}}\sum_{\substack{\bs{b} \mod{q}}}e_q(aF(\bs{b})+b_1r+\bs{b}'.\bs{c}')\rvert \\ 
&\quad \quad \quad \quad \times \lvert\sum_{c_1=1}^{\infty}\lambda(c_1)e_q(-rc_1)I_q(\bs{c})\rvert.
\end{split}
\end{equation*}
By Lemma ~\ref{lemmastationaryphase} we have
\begin{equation*}
\begin{split}
\sum_{c_1=1}^{\infty}\lambda(c_1)e_q(-rc_1)I_q(\bs{c})&= - \int \sum_{c_1 \leq x} \lambda(c_1)e_q(-rc_1) \frac{\partial}{\partial x}I_q(x,\bs{c}')\, dx \\
&\ll \frac{(r^{-1}|\bs{u}'|)^{\ve}|\bs{u}'|^{-\frac{1}{2}}}{qX}\int_{1}^X x^{\frac{3}{2}}\log x \, dx \, + \\
&\quad \quad \frac{(r^{-1}|\bs{u}'|)^{\ve}|\bs{u}'|^{-\frac{1}{2}}}{X}\int_{1}^X x^{\frac{1}{2}}\log x \, dx \\
&\ll_{\ve} \frac{X^{1+\ve}}{q^{\frac{1}{2}}} + q^{\frac{1}{2}}X^{\ve} \ll_{\ve} \frac{X^{1+\ve}}{q^{\frac{1}{2}}},
\end{split}
\end{equation*}
since $q \ll X$, and by using the bound $$\sum_{n \leq z} \lambda(n)e(\alpha n) \ll_f z^{\frac{1}{2}}\log z,$$ which is uniform in $\alpha$. Recall $T_q$ from ~\eqref{eq:tq}. We have shown
\begin{proposition}\label{pr2}
Suppose that $1 \leq |\bs{c}'| \ll X^{\ve}$ and $F^{-1}(0,\bs{c}') \neq 0$. Then,  
\begin{equation*}
N_q^{(1)}(\lambda, \bs{c}';X)\ll_{\ve} \left(\frac{X^{1+\ve}}{q^{\frac{3}{2}}}\right)T_q.
\end{equation*}
\end{proposition}

With Propositions ~\ref{pr1} and ~\ref{pr2} in place, we can complete our analysis of $N^{(1)}(\lambda;X)$. Let $1 \leq Y \ll X$ be a parameter to be chosen later. Then
\begin{equation*}
N^{(1)}(\lambda,\bs{c}';X) = \sum_{q \leq Y}q^{-3} N_q^{(1)}(\lambda, \bs{c}';X)+ \sum_{q > Y}q^{-3}N_q^{(1)}(\lambda,\bs{c}';X).
\end{equation*}
Using Proposition ~\ref{pr1} to estimate the sum up to $Y$, we get
\begin{equation*}
\begin{split}
\sum_{q \leq Y}q^{-3} N_q^{(1)}(\lambda, \bs{c}';X)\ll_{\ve} X^{\frac{1}{2}+\ve}\sum_{v \ll Y}v^{\frac{1}{2}}\sum_{u \ll Y/v}(u,F^{-1}(0,\bs{c}'))u^{-\frac{1}{2}} \ll_{\ve} (XY)^{\frac{1}{2}+\ve},
\end{split}
\end{equation*}
since $$\sum_{\substack{v \ll Y \\ p \mid v \implies p \mid N}} 1 \ll_{\ve} (NY)^{\ve}.$$ Applying Proposition ~\ref{pr2} to the second sum, we get by Lemma ~\ref{c'0} that
\begin{equation*}
 \sum_{q > Y}q^{-3}N_q^{(1)}(\lambda, \bs{c}';X)\ll_{\ve} X^{1+\ve}Y^{-\frac{1}{2}}.
\end{equation*}
The optimal choice for $Y$ is $Y = X^{\frac{1}{2}},$ and this gives us
\begin{equation*}\label{eq:n1}
N^{(1)}(\lambda;X) \ll_{\ve} X^{2-\frac{1}{4}+\ve}.
\end{equation*}
Combined with ~\eqref{eq:n0lambdax} this completes the proof of Theorem ~\ref{mainthm}.


\begin{thebibliography}{100}

\bibitem{BB} S. Baier and T. D. Browning, Inhomogeneous quadratic congruences.
{\em Funct. Approx.} {\bf 47} (2012), 267-286.

\bibitem{B} V. Blomer, Sums of Hecke eigenvalues over values of quadratic polynomials. {\em Int. Math. Res. Not. IMRN} \textbf{16} (2009): 29pp.

\bibitem{BH} V. Blomer and G. Harcos, Hybrid bounds for twisted $L$-functions. {\em J. reine angew. Math.\em}, \textbf{621} (2008):  53-79.

\bibitem{BMN} A. Booker, M. Milinovich and N. Ng, Subconvexity for modular forms $L$-functions in the $t$-aspect. {\em Advances in Mathematics} \textbf{341} (2019) 299-335.


\bibitem{BB1} R. de la Bret\`{e}che and T. D. Browning, Le probl\`{e}me des diviseurs pour des formes binaires de degr\'{e} 4. {\em J. reine angew. Math.} \textbf{646} (2010): 1-44.

\bibitem{BB2} R. de la Bret\`{e}che and T. D. Browning, Binary forms as sums of two squares and Ch\^{a}telet surfaces. {\em Israel J. Math.} \textbf{191} (2012): 973-1012.


\bibitem{BT} R. de la Bret\`{e}che and G. Tenenbaum, Moyennes de fonctions arithm\'{e}tiques de formes binaires. {\em Mathematika} \textbf{58} (2012): 290-304.

\bibitem{Br} T. D. Browning, The divisor problem for binary cubic forms. {\em J. Th\'{e}o. Nombres Bordeaux} \textbf{23} (2011): 579-602.

%\bibitem{BV} T. D. Browning and I. Vinogradov, Effective Ratner theorem for $SL(2,\mathbf{R}) \rtimes \mathbf{R}^2$ and gaps in $\sqrt{n}$ modulo $1$. {\em J. London. Math. Soc. (2)} \textbf{94} (2016): 61-84.


\bibitem{D} S. Daniel, On the divisor-sum problem for binary forms. {\em J. reine angew. Math.} \textbf{507} (1999):107-129.



\bibitem{DFI} W. Duke, J. Friedlander and H. Iwaniec, Bounds for automorphic $L$-functions. {\em Invent. Math.} \textbf{112} (1993): 1-8. 


\bibitem{FGKM} E. Fouvry, S. Ganguly, E. Kowalski and P. Michel, Gaussian distribution for the divisor function and Hecke eigenvalues in arithmetic progressions. {\em Comment. Math. Helv.} \textbf{89} (2014): 979-1014.



\bibitem{FI} J. Friedlander, H. Iwaniec, The polynomial $X^2+Y^4$ captures its primes. {\em Ann. Math.} \textbf{148} (1998): 945-1040.

\bibitem{G} J. R. Getz, Secondary terms in asymptotics for the number of zeros of quadratic forms over number fields. {\em J. London Math. Soc.} \textbf{98} (2018), no. 2: 275-305.

\bibitem{G2} G. Greaves, On the divisor-sum problem for binary cubic forms. {\em Acta. Arith.} \textbf{17} (1970): 1-28.


\bibitem{HB} D. R. Heath-Brown, A new form of the circle method, and its application to quadratic forms. {\em J. reine angew. Math.} \textbf{481} (1996): 149-206.
\bibitem{HB01} D. R. Heath-Brown, Primes represented by $x^3+2y^3$. {\em Acta. Math.} \textbf{186} (2001): 1-84. 

\bibitem{HB02} D. R. Heath-Brown, Linear relations amongst sums of two squares. {\em Number Theory and Geometry,} 133-176. London Math. Soc. Lecture Not. Ser. \textbf{303} CUP, 2003. 

\bibitem{HBP}D. R.  Heath-Brown and L. B.  Pierce, Simultaneous integer values of pairs of quadratic forms. {\em J. reine angew. Math.} \textbf{727} (2017): 85-143.

\bibitem{H} C. Hooley, On the number of divisors of quadratic polynomials. {\em Acta Math.} \textbf{110} (1963): 97-114.



\bibitem{IK} H. Iwaniec and E. Kowalski, {\em Analytic number theory}. American Math. Soc. Colloq. Pub. {\bf 53}, American Math.\ Soc., 2004.


\bibitem{KS} H. Kim, Functoriality for the exterior square of $GL_4$ and the symmetric fourth of $GL_2$ (with Appendix $1$ by D. Ramakrishnan and Appendix $2$ by H. Kim and P. Sarnak. {\em J. Amer. Math. Soc.} \textbf{16} (2003): 139-183. 

%\bibitem{K} V. Vinay Kumaraswamy, On correlations between class numbers of imaginary quadratic fields. {\em Acta. Arith,} to appear. Preprint available at \href{https://arxiv.org/abs/1702.04708}{ArXiv:1702.04708}. 

\bibitem{M} R. Munshi, The circle method and bounds for $L$-functions - I {\em Math. Ann.} \textbf{358} (2014): 389-401.

\bibitem{stein}
E. M. Stein,  {\em Harmonic analysis: real-variable methods, orthogonality, and oscillatory integrals}. 
Princeton University Press, 2016.

\bibitem{TT} N. Templier and J. Tsimerman, Non-split sums of coefficients of $GL(2)$-automorphic forms. {\em Israel J. Math.} \textbf{195} (2013): 677-723. 

\bibitem{TZ} K.-M Tsang and L. Zhao, On Lagrange's four squares theorem with almost prime variables. {\em J. reine angew. Math.} \textbf{726} (2017): 129-171.  

\bibitem{W44} G. N. Watson, {\em A treatise on the theory of {B}essel functions}, Cambridge Mathematical Library, Cambridge University Press, Cambridge, 1995, Reprint of the second (1944) edition (1995).


\end{thebibliography}
\end{document}